\documentclass[11pt]{amsart}
\usepackage{a4wide,amssymb}

\newcommand{\IN}{\mathbb N}
\newcommand{\w}{\omega}
\newcommand{\defeq}{\overset{\mathsf{def}}=}
\newcommand{\supp}{\mathrm{supp}}
\newcommand{\ddoteq}{\,\ddot{=}\,}
\newcommand{\e}{\varepsilon}

\newcommand{\C}{\mathcal C}
\newcommand{\cov}{\mathrm{cov}}
\newcommand{\A}{\mathcal A}

\newtheorem{theorem}{Theorem}[section]
\newtheorem{lemma}[theorem]{Lemma}
\newtheorem{claim}[theorem]{Claim}
\newtheorem{remark}[theorem]{Remark}
\newtheorem{problem}[theorem]{Problem}
\newtheorem{corollary}[theorem]{Corollary}

\newtheorem{question}[theorem]{Question}

\title[A $36$-Shelah group]{A  non-polybounded absolutely closed  $36$-Shelah group} 
\author{Taras Banakh}
\address{Ivan Franko National University of Lviv, Ukraine, and Jan Kochanowski University in Kielce, Poland}
\email{t.o.banakh@gmail.com}

\begin{document}
\begin{abstract} Let $\kappa$ be an infinite cardinal such that $\kappa^+=2^\kappa$. We prove that every group $H$ of cardinality $|H|\le\kappa$ is a subgroup of a group $G$ of cardinality $|G|=\kappa^+$ such that (i) $G$ is $36$-Shelah, which means that $A^{36}=G$ for any subset $A\subseteq G$ of cardinality $|A|=|G|$; (ii) $G$ is absolutely $\mathsf{T_{\!1}S}$-closed and projectively $\mathsf{T_{\!1}S}$-discrete, which means that for every homomorphism $h:G\to Y$ to a $T_1$ topological semigroup $Y$ the image $h[G]$ is a closed discrete subspace of $Y$, (iii) $G$ cannot be covered by finitely many algebraic subsets, i.e., subsets of the form $\{x\in G:xc_1xc_2\cdots xc_n=e\}$ for some $c_1,c_2,\cdots,c_n\in G$. 
\end{abstract}
\maketitle

\section{Introduction}

A semigroup $S$ is called
\begin{itemize}
\item {\em Jonsson} if every subsemigroup $X\subseteq S$ of cardinality $|X|=|S|$ coincides with $S$;
\item {\em $n$-Shelah} for $n\in\IN$ if $\{x_1\cdots x_n:x_1,\dots,x_n\in X\}=S$ for every subset $X\subseteq S$ of cardinality $|X|=|S|$;
\item {\em Shelah} if $S$ is $n$-Shelah for some $n\in\IN$.
\end{itemize}
It is clear that every finite semigroup is $1$-Shelah and every Shelah semigroup is Jonsson. A simple example of a countable Jonsson group is the Pr\"ufer $p$-group
$$\{z\in \mathbb C:\exists n\in\IN\;\;(z^{p^n}=1)\}$$where $p$ is a prime number.  On the other hand, by a result of Protasov \cite{Prot}, every countable Shelah semigroup is finite. 

The first example of an infinite Shelah semigroup was constructed by Shelah \cite{Shelah}. More precisely, for every infinite cardinal $\kappa$ with $\kappa^+=2^\kappa$, Shelah has constructed a $6640$-Shelah group of cardinality $\kappa^+=2^\kappa$. A valuable additional property of the Shelah group $G$ is its nontopologizability, which means that $G$ admits no non-discrete group topology. The Shelah's group was the first example of a non-topologizable infinite group, which answered a famous question of Markov \cite{Markov}. A countable non-topologizable group was constructed by Ol'shanskii \cite{Ol}.

In this paper we elaborate the ideas of Shelah \cite{Shelah} in more details and construct for every infinite cardinal $\kappa$ with $\kappa^+=2^\kappa$ a $36$-Shelah group of cardinality $\kappa^+$ which is projectively $\mathsf{T_{\!1}S}$-discrete, absolutely $\mathsf{T_{\!1}S}$-closed, and  not polybounded.

Let us recall the definitions of those notions.

Let $\C$ be a class of topological semigroups. A semigroup $X$ is called
\begin{itemize}
\item {\em  $\C$-discrete} if for every injective homomorphism $h:X\to Y$ to a topological semigroup $Y\in\C$ the image $h[X]$ is a discrete subspace of $Y$;
\item {\em projectively $\C$-discrete} if  for every homomorphism $h:X\to Y$ to a
topological semigroup $Y\in\C$ the image $h[X]$ is a discrete subspace of $Y$;
\item {\em absolutely $\C$-closed} if   for every homomorphism $h:X\to Y$ to a
topological semigroup $Y\in\C$ the image $h[X]$ is a closed subspace of the topological semigroup $Y$;
\item {\em injectively $\C$-closed} if   for every injective  homomorphism $h:X\to Y$ to a
topological semigroup $Y\in\C$ the image $h[X]$ is a closed subspace of the topological semigroup $Y$;
\item {\em $\C$-closed} if $X$ is closed in every topological semigroup $Y\in\C$ containing $X$ as a discrete subsemigroup.
\end{itemize}
It is clear that every absolutely $\C$-closed semigroup is injectively $\C$-closed and every injectively $\C$-closed semigroup is $\C$-closed.

Observe that a group is non-topologizable if and only if it is $\mathsf{TG}$-discrete for the class $\mathsf{TG}$ of Hausdorff topological groups. Let $\mathsf{T_{\!1}S}$ denote the class of $T_1$ topological semigroups, i.e., topological semigroups whose all finite subsets are closed. Since $\mathsf{TG}\subseteq\mathsf{T_{\!1}S}$, every (projectively) $\mathsf{T_{\!1}S}$-discrete group is (projectively) $\mathsf{TG}$-discrete and hence non-topologizable. 

By \cite[Proposition 3.2]{BB}, a semigroup is injectively $\mathsf{T_{\!1}S}$-closed if and only if it is $\mathsf{T_{\!1}S}$-closed and $\mathsf{T_{\!1}S}$-discrete. By \cite[Proposition 3.3]{BB}, every absolutely $\mathsf{T_{\!1}S}$-closed semigroup is projectively $\mathsf{T_{\!1}S}$-discrete. In particular, every absolute $\mathsf{T_{\!1}S}$-closed group is non-topologizable.

For a semigroup $X$ let $X^1=X\cup\{1\}$ where $1$ is an element such that $x1=x=1x$ for all $x\in X^1$. A function $f:X\to X$ on a semigroup $X$ is called a {\em semigroup polynomial} if there exist $n\in\IN$ and elements $c_0,\dots,c_n\in X^1$ such that $f(x)=c_0xc_1\cdots xca_n$ for all $x\in X$. 

A semigroup $X$ is defined to be {\em polybounded} if $X=\bigcup_{i=1}^np_i^{-1}(b_i)$ for some elements $b_1,\dots,b_n\in X$ and some semigroup polynomials $p_1,\dots,p_n$ on $X$. 

Polybounded semigroups play an important role in the theory of categorically closed semigroups \cite{BB}, \cite{ACS} and have many nice properties. For example, cancellative polybounded semigroups are groups \cite[1.8]{BB}, polybounded paratopological groups are topological groups \cite[8.1]{BB}, each polybounded group $X$ is $\mathsf{T_{\!1}S}$-closed \cite[6.1]{BB}. Moreover, by \cite[1.9]{BB}, a countable group $X$ is $\mathsf{T_{\!1}S}$-closed if and only if it is polybounded. The latter  characterization is specific for countable groups and does not extend to the uncountable case, as shown by the following theorem which is the main result of this paper.

\begin{theorem}\label{t:main}  For every infinite cardinal $\kappa$ with $\kappa^+=2^\kappa$ there exists a non-polybounded  absolutely $\mathsf{T_{\!1}S}$-closed $36$-Shelah group $X$ of cardinality $|X|=\kappa^+$.
\end{theorem}

The proof of Theorem~\ref{t:main} follows the lines of the proof of Shelah's Theorem 2.1 \cite{Shelah}, who applied is his contruction the small cancellation theory for free products with amalgamation.  The necessary definitions and results of this theory are recalled in Section~\ref{s:SCT}. In Section~\ref{s:malnormal} we prove a lemma on the preservation of malnormality by suitable quotients of free products with amalgamation. In Section~\ref{s:C'} we prove the condition $C'(\lambda)$ for some special subsets of a free product with amalgamation. In Section~\ref{s:AL} we use those special subsets in the proofs of two Amalgamation Lemmas \ref{l:amalgamation} and \ref{l:amalgamation2}. In Section~\ref{s:AL} the Amalgamation Lemmas are applied in the inductive step of the inductive proof of Embedding Lemma~\ref{l:EL}, which is  used in the inductive step of the proof of  Theorem~\ref{t:shelah}, which is a more elaborated version of Theorem~\ref{t:main}. In the final Section~\ref{s:final} we ask some question related to $n$-Shelah groups.

\section{Preliminaries}

Every ordinal $\alpha$ coincides with the set $\{\gamma:\gamma<\alpha\}$ of smaller ordinals. In particular, every number $n\in\IN$ coincides with the set $\{0,\dots,n-1\}$. 

For an ordinal $\alpha$ its {\em cofinality} $\mathrm{cf}(\alpha)$ is the smallest cardinality of a subset $C\subseteq \alpha$ such that for every $a\in\alpha$ there exists $c\in C$ such that $a\le c$.

Cardinals are the smallest ordinals of a given cardinality. The successor cardinal of a cardinal $\kappa$ is denoted by $\kappa^+$. We denote by $\w$ the smallest infinite cardinal and by $\IN$ the set $\w\setminus\{0\}=\{1,2,3,\dots\}$ of positive finite ordinals.

 Every number $n\in\IN$ is endowed with the binary operations $\oplus:n\times n\to n$ and $\ominus:n\times n\to n$ defined by
$$i\oplus j=\begin{cases}i+j&\mbox{if $i+j<n$},\\
i+j-n&\mbox{if $i+j\ge n$},
\end{cases}
\quad\mbox{and}\quad
i\ominus j=\begin{cases}i-j&\mbox{if $j\le i$},\\
i-j+n&\mbox{if $j>i$},
\end{cases}
$$ for $i,j\in n$. So, $\oplus$ and $\ominus$ are the operations of addition and subtraction modulo $n$. For an element $x\in n$ let $\ominus x=0\ominus x$ be the inverse element to $x$ in the group $(n,\oplus)$.

In the sequel by $e$ we shall denote the identity element of a group.


A subset $A$ of a semigroup $X$ is called {\em algebraic} if $A=p^{-1}(b)$ for some $b\in X$ and some semigroup polynomial $p:X\to X$. For a semigroup $X$ we denote by $\A_X$ the family of all algebraic subsets in $X$. For a subset $B$ of a semigroup $X$ its {\em polyboundedness number} $\cov(B;\A_X)$ is defined as the smallest cardinality $|\mathcal B|$ of a subfamily $\mathcal B\subseteq\A_X$  such that $B\subseteq\mathcal B$. It is clear that a semigroup $X$ is polybounded if and only if $\cov(X;\A_X)<\w$. A subset $B$ of a semigroup $X$ is defined to be {\em polybounded in $X$} if $\cov(B;\A_X)$ is finite.

The following lemma was proved in \cite[Lemma 6.1]{BB}.

\begin{lemma}\label{l:polybounded} Every polybounded subset $B$ of a semigroup $X$ is closed in every $T_1$-topological semigroup $Y$ containing $X$ as a discrete subsemigroup.
\end{lemma}

We recall that a {\em topological semigroup} is  topological space endowed with a continuous associative binary operation. Equivalently, a topological semigroup can be defined as a semigroup endowed with a semigroup topology. A topology $\tau$ on a semigroup $X$ is called a {\em semigroup topology} if $(X,\tau)$ is a topological semigroup.

For a semigroup $X$ its {\em zero-extension} is the semigroup $X^0=X\cup\{0\}$ where $0\notin X$ is an element such that $0x=0=x0$ for all $x\in X^0$.
A semigroup $X$ is called {\em zero-closed} if $X$ is closed in $X^0$ endowed with any Hausdorff semigroup topology. It is clear that every absolutely $\mathsf{T_{\!1}S}$-closed semigroup is zero-closed.  The following lemma is proved in \cite{BR}.

\begin{lemma}\label{l:zero-closed} Every zero-closed infinite semigroup $X$ has $\cov(X;\A_X)<|X|$.
\end{lemma}

The proof of the following lemma can be found in \cite[Proposition 3.3]{BB}.

\begin{lemma}\label{l:ac=>pd} Every absolutely $\mathsf{T_{\!1}S}$-closed semigroup is projectively $\mathsf{T_{\!1}S}$-discrete.
\end{lemma}

\section{Small cancellation theory}\label{s:SCT}

In this section we recall some notions and results of the small cancellation theory in free products with amalgamation, following \cite[\S1]{Shelah}, \cite[\S IV.1]{Hesse} and \cite[\S V.11]{LS}.

Let $G,H,L$ be groups with $G\cap L=H$. Given two elements $x,y\in G\cup L$ we write $x\sim y$ if $\{x,y\}\subseteq G$ or $\{x,y\}\subseteq L$. Let $F$ be the free group generated by the set $L\cup G$ and $N_*$ be the smallest normal subgroup of $F$, containing  the set $$\{xyz^{-1}:x,y,z\in L\cup G\;\wedge\;x\sim y\;\wedge\;z=xy\}.$$The quotient group $F_*=F/N_*$ is called the {\em free product with amalgamation of $L$ and $G$ over $H$} and is denoted by $L*_HG$. 

By \cite[Fact 1.2]{Shelah}, the function $L\cup G\to F_*$, $x\mapsto x N_*$, is  injective, which allows us to identify the groups $L,G$ with their images in $F_*$. By \cite[Fact 1.2]{Shelah}, the intersection of the groups $L,G$ in $F_*$ coincides with the canonical image of the group $H=L\cap G$ in $F_*$. 


Every element $x\in F_*$ can be written as the product $x=x_1\cdots x_n$ of some elements $x_1,\dots,x_n\in L\cup G$ such that $x_i\not\sim x_{i+1}$ for all $i<n$. Such a representation is called a {\em canonical representation} of $x$. The following lemma from \cite[Fact 1.3]{Shelah} shows that a canonical representation of an element is uniquely determined up to multipliers from the subgroup $H$.

\begin{lemma}\label{l:canrep} If $x_1\cdots x_n$ and $y_1\cdots y_m$ are two canonical representations of the  same element in $F_*$, then $n=m$ and there exists a sequence $(h_i)_{i=0}^n$ in $H$ such that $h_0=e=h_n$ and $x_i=h^{-1}_{i-1}y_ih_i$ for all $i\in\{1,\dots,n\}$.
\end{lemma}

For an element $x\in F_*$ its {\em length} $|x|$ is defined as the length $n$ of any canonical representation $x_1\cdots x_n$ of $x$. The length of the identity element $e$ of $F_*$ is  zero. Lemma~\ref{l:canrep} ensures that the length of any element $x\in F_*$ is well-defined.  The length satisfies the triangle inequality
$$|xy|\le|x|+|y|$$for every $x,y\in F_*$. 

Given elements $w,u_1,\dots,u_n$ of $F_*$, we write $w\equiv u_1\cdots u_n$ if $w=u_1\cdots u_n$ in $F_*$ and $|w|=|u_1|+\cdots+|u_n|$.  

A element $x\in F_*$ is called {\em weakly cyclically reduced} (resp. {\em cyclically reduced\/}) if either $|x|\le 1$ or $x$ has a canonical representation $x_1\cdots x_n$ such that $x_1x_n\notin H$ (resp. $x_1\not\sim x_n$). Observe that an element  $x\in F_*$ of length $|x|>1$ is cyclically reduced if and only if it has even length. 

A set $R\subseteq F_*$ is called {\em half-symmetrized} if every element of $R$ is weakly cyclically reduced and every weakly cyclically reduced element of the set $\{crc^{-1}:c\in F_*,\;r\in R\}$ belongs to $R$. A subset $R\subseteq F_*$ is called {\em symmetrized} if it is half-symmerized and $\{r^{-1}:r\in R\}\subseteq R$.

For a subset $R\subseteq F_*$ its ({\em half-}){\em symmetrized hull} is  the smallest (half-)symmetrized set containing $R$. The following lemma describes the structure of the half-symmetrized hull of a cyclically reduced element. 

\begin{lemma}\label{l:half-hull} Let $r\in F_*$ be a cyclically reduced element of length $n\ge 2$ and $x_0\cdots x_{n-1}$ be a canonical representation of $r$. The half-symmetrized hull of the set $\{r\}$ coincides with the set 
$$R^\circ=\{cx_ix_{i\oplus 1}\cdots x_{i\oplus(n-1)}c^{-1}:i\in n,\;c\in (L\cup G)\setminus Hx_i^{-1},\;c\sim x_i\}.$$
\end{lemma}

\begin{proof} Recall that $\oplus$ is the operation of addition modulo $n$ on the set $n=\{0,\dots,n-1\}$. 

\begin{claim}\label{cl:wcr} Every element $w$ of the set $R^\circ$ is weakly cyclically reduced.
\end{claim}

\begin{proof} By the definition of $R^\circ$, there exist $i\in n$ and $c\in (L\cup G)\setminus Hx_i^{-1}$ such that $c\sim x_i$ and $w=cx_ix_{i\oplus 1}\cdots x_{i\oplus(n-1)}c^{-1}$. It follows that  $cx_i\sim x_i\not\sim x_{i\oplus1}$ and $cx_i\in (L\cup G)\setminus H$. 

If $c\notin H$, then $x_{i\oplus(n-1)}\not\sim x_{i}\sim c$ implies $x_{i\oplus(n-1)}\not\sim c$ and hence $(cx_i)x_{i\oplus1}\cdots x_{i\oplus(n-1)}c^{-1}$ is a canonical representation of $r'$. Since $c^{-1}(cx_i)=x_i\notin H$, the element $r'$ is weakly cyclically reduced. 

If $c\in H$, then the cyclic reducibility of $r$ ensures that $cx_i\sim x_i\not\sim x_{i\oplus(n-1)}\sim x_{i\oplus(n-1)}c^{-1}$ and hence the element $r'=(cx_i)x_{i\oplus1}\cdots x_{i\oplus(n-2)}(x_{i\oplus(n-1)}c^{-1})$ is (weakly) cyclically reduced.
\end{proof}

\begin{claim}\label{cl:crc} Every element $r^\circ\in R^\circ$ equals $crc^{-1}$ for some $c\in F_*$.
\end{claim}

\begin{proof} Given an element $r^\circ\in R^\circ$, find $i\in n$ and $y\in (L\cup G)\setminus Hx_i^{-1}$ such that $r^\circ =yx_ix_{i\oplus1}\cdots x_{i\oplus(n-1)}y^{-1}$. Observe that for the elements $d=x_0\cdots x_{i\oplus(n-1)}=x_0\cdots x_{i\ominus1}$ and $c=yd^{-1}\in F_*$ we have 
\begin{multline*}
r^\circ =yx_ix_{i\oplus 1}\cdots x_{i\oplus(n-1)}y^{-1}=y(x_0\cdots x_{i\ominus1})^{-1}x_0\cdots x_n(x_0\cdots x_{i\oplus(n-1)})y^{-1}=\\
yd^{-1}rdy^{-1}=crc^{-1}.
\end{multline*}
\end{proof}

\begin{claim}\label{cl:hs} The set $R^\circ$ is half-symmetrized. 
\end{claim}

\begin{proof}  By Claim~\ref{cl:wcr}, every element of $R^\circ$ is weakly cyclically reduced. Given any weakly cyclically reduced element $w\in\{crc^{-1}:c\in F_*,\; r\in R^\circ\}$, it remains to prove that $w\in R^\circ$.  Let $a\in F_*$ be an element of the smallest possible length such that $w=ax_{i}x_{i\oplus1}\cdots x_{i\oplus(n-1)}a^{-1}$ for some $i\in n$. Such an element $a$ exists by Claim~\ref{cl:crc}.

We claim that $|a|\le1$. To derive a contradiction, assume that $|a|>1$ and let $a_1\cdots a_k$ be a canonical representation of $a$. If $a_k\sim x_i$, then $a_k^{-1}\not\sim x_{i\ominus 1}=x_{i\oplus(n-1)}$, by the cyclic reducibility of $r$. The weak cyclic reducibility of $$w=a_1\cdots a_{k-1}(a_kx_i)x_{i\oplus1}\cdots x_{i\oplus(n-1)}a_k^{-1}\cdots a_1^{-1}$$ implies that $a_kx_i\in H$. Observe that the element $ax_i=a_1\dots a_{k-2}(a_{k-1}a_kx_i)$ has length $|ax_i|=k-1=|a|-1$ and $(ax_i)x_{i\oplus1}\cdots x_{i\oplus(n-1)}x_i(ax_i)^{-1}=w$, which contradicts the minimality of $|a|$. This contradiction shows that $|a|\le 1$ and hence $a\in L\cup G$.

If $a\sim x_i$ and $ax_i\notin H$, then $w=ax_ix_{i\oplus1}\cdots x_{i\oplus(n-1)}a^{-1}\in R^\circ$ by the definition of $R^\circ$.

If $a\sim x_i$ and $ax_i\in H$, then $ax_i\sim x_{i\oplus 1}\notin H$ and $ax_ix_{i\oplus1}\notin H$. Then $$w=(ax_i)x_{i\oplus 1}\cdots x_{i\oplus (n-1)}x_i(ax_i)^{-1}\in R^\circ$$ by the definition of $R^\circ$.

If $a\not\sim x_i$, then $a\notin H$ and $a\sim x_{i\oplus(n-1)}=x_{i\ominus1}$ by the cyclic reducibility of $r=x_0\cdots x_{n-1}$. Then $w=(ax_{i\ominus 1}^{-1})x_{i\ominus1}x_i\cdots x_{i\oplus(n-2)}(ax_{i\ominus1}^{-1})^{-1}$ and $ax_{i\ominus 1}^{-1}\sim x_{i\ominus 1}$. Since $(ax_{i\ominus1}^{-1})x_{i\ominus1}=a\notin H$, the element $w=(ax_{i\ominus 1}^{-1})x_{i\ominus1}x_i\cdots x_{i\oplus(n-2)}(ax_{i\ominus1}^{-1})^{-1}$ belongs to $R^\circ$ by the definition of $R^\circ$.
\end{proof} 

Claims~\ref{cl:wcr}, \ref{cl:crc} and \ref{cl:hs} imply that the set $R^\circ$ coincides with the half-symmetrized hull of $\{r\}$.
\end{proof}

Lemma~\ref{l:half-hull} implies the following description of the symmetrized hull of a cyclically reduced element, which can be compared with  \cite[Claim 1.7]{Shelah}.

\begin{corollary}\label{c:sym-hull} Let $r\subseteq L*_HG$ be a cyclically reduced element of length $n\ge 2$ and $x_1\cdots x_n$ be a canonical representation of $r$. The symmetrized hull of the set $\{r\}$ consists of the elements $$cx_i x_{i\oplus1}\cdots x_{i\oplus(n-1)}c^{-1}\quad\mbox{and}\quad dx_{i}^{-1} x_{i\ominus1}^{-1}\cdots x_{i\ominus(n-1)}^{-1}d^{-1}$$ where $i\in n$, $c\in (L\cup G)\setminus Hx_i^{-1}$, $c\sim x_i$, and $d\in(L\cup G)\setminus Hx_i$, $d\sim x_i^{-1}$.
\end{corollary}

Given elements $a,b,c\in F_*$, we write $a\ddoteq bc$ if $a=bc$ in $F_*$ and $|b|+|c|-1\le|a|\le|b|+|c|$.  This happens if and only if $a=bc$ and $b_nc_1\notin H$ for some (equivalently, every) canonical representations $b_1\cdots b_n$ and $c_1\cdots c_m$ of the elements $b$ and $c$, respectively.

Let $R$ be a subset of $F_*$. An element $p\in F_*$ is called an {\em $R$-piece of an element} $r\in F_*$ if there exists a element $q\in R\setminus\{r\}$ such that $r\ddoteq px$ and $q\ddoteq py$ for some elements $x,y\in F_*$.

Let $\lambda<1$ be a positive real number.  A subset $R\subseteq A$ is defined to satisfy {\em the condition} $C'(\lambda)$ if $\frac1\lambda<\min\{|r|:r\in R\}$ and every $R$-piece $p$ of every element $r\in R$ has length $|p|<\lambda |r|$. The following lemma is one of principal results of the small cancellation theory and can be found in \cite[11.2]{LS}, \cite[1.9]{Shelah} or \cite[Haupt-Lemma IV.1.4]{Hesse}. 

\begin{lemma}\label{l:large-piece} Let $N$ be the normal subgroup of $F_*$ generated by a symmetrized set $R\subseteq F^*$ satisfying the condition $C'(\lambda)$ for some $\lambda\le\frac16$. For every $w\in N\setminus\{e\}$ there exist elements $u,s,v,t\in F_*$ and a cyclically reduced element $r\in R$ such that $w\equiv usv$, $r\equiv st$ and $|s|>(1-3\lambda)|r|$.
\end{lemma}

\begin{lemma}\label{l:qLG} Let $R\subseteq F_*$ be a symmetrized subset satisfying the condition $C'(\lambda)$ for some positive $\lambda\le\frac16$. Let  $N$ be the smallest normal subgroup containing $R$, $F_{**}=F_*/N$ be the quotient group and $q:F_*\to F_{**}$ is a quotient homomorphism. Then
\begin{enumerate}
\item every element $w\in N\setminus\{e\}$ has length $|w|>\frac1\lambda-3\ge 3$;
\item the restriction $q{\restriction}_{L\cup G}$ is injective;
\item $q[H]=q[L]\cap q[G]$.
\end{enumerate}
\end{lemma}

\begin{proof} 1. Given any element $w\in N\setminus\{e\}$, apply Lemma~\ref{l:large-piece} and find elements $u,s,v,t\in F_*$ and a cyclically reduced element $r\in R^\circ$ such that $w\equiv usv$, $r\equiv st$ and $|s|>(1-3\lambda)|r|$. Taking into account that $|r|>1/\lambda$, we conclude that
$$|w|\ge |s|>(1-3\lambda)|r|>(1-3\lambda)/\lambda=\tfrac1\lambda-3\ge 6-3=3.$$
\smallskip

2. To show that the restriction $q{\restriction}_{L\cup G}$ is injective, take any  elements $x,y\in L\cup G$. Assuming that $xN=q(x)=q(y)=yN$, we conclude that $y^{-1}x\in N$ and hence $y^{-1}x=e$ because $|y^{-1}x|\le 2<3$. This shows that $q{\restriction}_{L\cup G}$ is injective.
\smallskip

3. It is clear that $q[H]=q[L\cap G]\subseteq q[L]\cap q[G]$. To show that $q[L]\cap q[G]$, take any element $y\in q[L]\cap q[G]$ and find elements $l\in L$ and $g\in G$ such that $lN=y=gN$. Then $g^{-1}l\in N$. Since $|g^{-1}l|\le 2<3$, the first item ensures that $g^{-1}l=e$ and hence $g=l\in L\cap G=H$ and $y=q[l]\in q[H]$.
\end{proof}

\section{The malnormality in free products with amalgamation}\label{s:malnormal}

Let $H$ be a subgroup of a group $G$. An element $g\in G$ is called {\em $H$-malnormal} if $g\notin H$ and $H\cap gHg^{-1}=\{e\}$ where $e$ is the identity element of the group $H\subseteq G$. A subgroup $H$ of a group $G$ is called {\em malnormal in $G$} if every element $x\in G\setminus H$ is $H$-malnormal.

Let $L,H,G$ be groups such that $L\cap G=H$, and let $F_*=L*_HG$ be the free product of $L$ and $G$ with amalgamation $H$. 

An element $r\in F_*\setminus\{e\}$ of even length $n=|r|$ is called {\em half $H^-$-separated} if there exist a canonical representation $x_0\cdots x_{n-1}$ of $r$ and $\e\in\{0,1\}$ such that for every distinct even numbers $i,j\in n$ we have $x_{\e\oplus i}^{-1}\notin Hx_{\e\oplus j}H$.
A subset $R\subseteq F_*$ is called {\em half $H^-$-separated} if each element $r\in R$ has positive even length and is half $H^-$-separated.

\begin{lemma}\label{l:malnormal} Let $L,G,H$ be groups with $L\cap G=H$ and $R\subseteq F*_H G$ be a half $H^-$-separated set whose symmetrized hull $R^\circ$ satisfies the condition $C'(\lambda)$ for some positive $\lambda<\frac16$ such that $(1-6\lambda)|r|>4$ for every cyclically reduced element $r\in R^\circ$. Let $N$ be the smallest normal subgroup of $F_*=L*_HG$ containing the set $R$.  If the subgroup $H$ is malnormal in $L$, then for every $x\in F_*\setminus GN$ and $y\in G\setminus\{e\}$ we have $xyx^{-1}\notin GN$.
\end{lemma}

\begin{proof} To derive a contradiction, assume that there exists $x\in F_*\setminus GN$ such that $xyx^{-1}\in zN$ for some $y,z\in G\setminus\{e\}$. We can assume that $x$ has the smallest possible length. 

\begin{claim}\label{cl:7} Every $w\in N\setminus \{e\}$ has length $|w|\ge 7$.
\end{claim}

\begin{proof} By Lemma~\ref{l:large-piece}, there exist elements $u,s,v,t\in F_*$ and a cyclically reduced element $r\in R^\circ$ such that $w\equiv usv$, $r\equiv st$, $|s|>(1-3\lambda)|r|$. Since $R^\circ$ satisfies the condition $C'(\lambda)$, $|r|>1/\lambda>6$.  Since the length of the cyclically reduced element $r$ is an even number, $|r|\ge 8$. By our assumption, $(1-6\lambda)|r|>4$ and hence
$$|w|\ge |s|>(1-3\lambda)|r|=\tfrac12|r|+\tfrac12(1-6\lambda)|r|>4+\tfrac124=6.$$
\end{proof}

\begin{claim} $|x|\ge 3$.
\end{claim}

\begin{proof} Assuming that $|x|\le 2$, we conclude that the element $xyx^{-1}z^{-1}\in N$ has length $|xyx^{-1}z^{-1}|\le 2|x|+2\le 6$, which contradicts Claim~\ref{cl:7}.
\end{proof}


Let $x_1\cdots x_n$ be a canonical representation of $x$. 

\begin{claim} $x_n\in L\setminus H$.
\end{claim}

\begin{proof} Assuming that $x_n\in G\setminus H$, we conclude that the element $\check y=x_nyx_n^{-1}$ belongs to $G\setminus \{e\}$. Assuming that the element $\check x=x_1\cdots x_{n-1}$ belongs to $GN$, we conclude that $x=\check xx_n\in GNG=GGN=GN$, which contradicts the choice of $x$. This contradiction shows that $\check x\notin GN$. Since $\check x\check y\check x^{-1}=xyx^{-1}\in zN$ and $|\check x|=n-1<|x|$, we obtains a contradiction with the minimality of the length of $x$. 
\end{proof}

\begin{claim}\label{cl:x1} $x_1\in L\setminus H$.
\end{claim}

\begin{proof} Assuming that $x_1\notin L\setminus H$, we conclude that $x_1\in G\setminus H$. Consider the element $\hat x=x_2\dots x_n$. Assuming that $\hat x\in GN$, we conclude that $x=x_1\hat x\in GGN=GN$, which contradicts the choice of $x$. This contradiction shows that $\hat x\notin GN$. Taking into account that $\hat xy\hat x^{-1}=x_1^{-1}xyx^{-1}x_1\in x_1^{-1}zx_1N$, $x_1^{-1}zx_1\in G\setminus\{e\}$ and $|\hat x|=n-1<|x|$, we obtains a contradiction with the minimality of $|x|$. This contradiction shows that $x_1\in L\setminus H$.
\end{proof}

Let $$k=\begin{cases}
n&\mbox{if $y\notin H$},\\
n-1&\mbox{if $y\in H$},
\end{cases}
$$ and
$$\check y=\begin{cases}y&\mbox{if $y\notin H$},\\
x_nyx_{n}^{-1}&\mbox{if $y\in H$}.
\end{cases}
$$
The $H$-malnormality of $x_n\in L\setminus H$ implies that $x_1\dots x_k\check yx_k^{-1}\cdots x_1^{-1}$ is a canonical representation of the element $xyx^{-1}$. If $z\notin H$, then Claim~\ref{cl:x1} ensures that $x_1\dots x_k\check y x_k^{-1}\cdots x_1^{-1}z^{-1}$ is a canonical representation of $xyx^{-1}z^{-1}$ witnessing that $|xyx^{-1}z^{-1}|=2k+2$. If $z\in H$, then  $x_1\dots x_k\check y x_k^{-1}\cdots x_2^{-1}(x_1^{-1}z^{-1})$ is a canonical representation of $xyx^{-1}z^{-1}$, witnessing that $|xyx^{-1}z^{-1}|=2k+1$.

Since $xyx^{-1}z^{-1}\in N$, we can apply Lemma~\ref{l:large-piece} and find elements $u,s,v,t\in F_*$ and a cyclically reduced element $r\in R^\circ$ such that $xyx^{-1}z^{-1}\equiv usv$, $r\equiv st$, and $|s|>(1-3\lambda)|r|$. 

By our assumption, $(1-6\lambda)|r|>4$ and hence $$3\lambda|r|<\tfrac12|r|-2.$$


\begin{claim}\label{cl:more} $|u|+|s|\ge k+3$.
\end{claim}

\begin{proof} To derive a contradiction, assume that $|u|+|s|\le k+2$. Write $s$ as $s\equiv s's''$ for some $s',s''\in F_*$ such that $|u|+|s'|=\min\{k,|u|+|s|\}$ and hence $|s''|\le 2$. It follows from $s's''t\equiv st\equiv r\in N$ that  $s'\in (s''t)^{-1}N$ and 
$$|s''t|=|r|-|s'|=|r|-|s|+|s''|< |r|-(1-3\lambda)|r|+2=3\lambda|r|+2<(\tfrac12|r|-2)+2=\tfrac12|r|.$$ The element $r$ of length $|r|>\frac1\lambda>6$ is cyclically reduced and hence has even length $|r|\ge8$. Then the inequality $|s''t|<\tfrac12|r|$ implies $|s''t|\le \tfrac12|r|-1$ and hence  $$|s'|=|r|-|s''t|\ge |r|-(\tfrac12|r|-1)=\tfrac12|r|+1\ge |s''t|+2.$$
By Lemma~\ref{l:canrep}, the inequality $|u|+|s'|\le k\le|x|$ and the equality  $xyx^{-1}z^{-1}\equiv us's''v$  imply that $x_1\cdots x_n=x\equiv us'\hat x$ for some $\hat x\in F_*$. It follows from $s'\in (s''t)^{-1}N$ that $x\equiv us'\hat x\in u(s''t)^{-1}\hat xN$ and hence $u(s''t)^{-1}\hat xy\hat x^{-1}(s''t)u^{-1}\in xyx^{-1}N=zN$. Assuming that $u(s''t)^{-1}\hat x\in GN$, we conclude that $x=us'\hat x\in u(s''t)^{-1}\hat xN\in GNN=GN$, which contradicts the choice of $x$. Finally, observe that $|u(s''t)^{-1}\hat x|\le|u|+|s''t|+|\hat x|\le  |u|+(|s'|-2)+|\hat x|=|us'\hat x|-2=|x|-2$, which contradicts the minimality of $|x|$.
\end{proof}

\begin{claim}\label{cl:less} $|u|\le k-2$.
\end{claim}

\begin{proof} To derive a contradiction, assume that $|u|\ge k-1$. Since $|s|>(1-3\lambda)|r|=\frac12|r|+\frac12(1-6\lambda)|r|>4+2=6$ and $x_1\cdots x_k\check yx_k^{-1}\cdots x_1z^{-1}\equiv usv$, we can write $s$ as $s\equiv \sigma s'$ where $|u|+|\sigma|=\max\{k+1,|u|\}$ and hence $|\sigma|\le 2$. It follows from $\sigma s't=st=r\in N$ that $s'\in (t\sigma)^{-1}N$. Observe that
\begin{multline*}
|t\sigma|\le |t|+|\sigma|+|s'|-|s'|=|\sigma s't|-|s'|=|r|-(|s|-|\sigma|)=|r|-|s|+|\sigma|<\\
|r|-(1-3\lambda)|r|+2=3\lambda|r|+2<(\tfrac12|r|-2)+2=\tfrac12|r|.
\end{multline*}
Since the $r$ has even length $|r|$, the strict inequality $|t\sigma|<\tfrac12|r|$ implies $|t\sigma|\le\tfrac12|r|-1$. The cyclic reducibility of $r\equiv \sigma s't$ implies that $|s'|+|t\sigma|=|r|$ and hence  
$$|s'|=|r|-|t\sigma|\ge |r|-\tfrac12|r|+1=\tfrac12|r|+1\ge|t\sigma|+2.$$
It follows from $x_1\cdots x_k\check yx_k^{-1}\cdots x_1^{-1}z^{-1}=xyx^{-1}z^{-1}\equiv u\sigma s'v$ and $|u\sigma|=\max\{k+1,|u|\}$ that $x^{-1}z^{-1}=x_n^{-1}\cdots x_1^{-1}z^{-1}\equiv \hat x s'v$ for some $\hat x\in F_*$.
Consider the element $\check x=\hat x(t\sigma)^{-1}v$ of $F_*$.
 It follows that $\check x=\hat x(t\sigma)^{-1}v\in \hat x s'N v=\hat xs'vN=x^{-1}z^{-1}N$ and hence $\check x^{-1}y\check x\in zxyx^{-1}z^{-1}N=zzNz^{-1}N=zN$. Assuming that $\check x\in GN$, we conclude that $x^{-1}\in \check x zN\in GNGN=GN$, which contradicts the choice of $x$. This contradiction shows that $\check x\notin GN$. 
Finally, observe that $$|\check x|\le|\hat x|+|t\sigma|+|v|\le |\hat x|+|s'|-2+|v|=|\hat xs'v|-2=|x^{-1}z^{-1}|-2\le |x|+1-2<|x|,$$
which contradicts the minimality of $|x|$.
\end{proof}

By Claims~\ref{cl:more} and \ref{cl:less}, $|u|\le k-2$ and $|u|+|s|\ge k+3$.  Write $s$ as $s\equiv s'\sigma s''$ for some $s',\sigma, s''\in F_*$ such that $|u|+|s'|=k-2$ and $|u|+|s'|+|\sigma|=k+3$. The equalities $x_1\cdots x_k\check yx_k^{-1}\cdots x_1^{-1}z^{-1}=xyx^{-1}z^{-1}\equiv usv\equiv us'\sigma s''v$ and Lemma~\ref{l:canrep} imply that $\sigma=h_0x_{k-1}x_k\check yx_k^{-1}x_{k-1}^{-1}h_5$ for some $h_0,h_5\in H$. Let $\sigma=\sigma_1\sigma_2\sigma_3\sigma_4\sigma_5$ be a canonical representation of $\sigma$.  Lemma~\ref{l:canrep} implies that $\sigma_1=h_0x_{k-1}h_1^{-1}$, $\sigma_2=h_1x_kh_2^{-1}$, $\sigma_3=h_2\check yh_3^{-1}$, $\sigma_4=h_3x_k^{-1}h_4^{-1}$, $\sigma_5=h_4x_{k-1}^{-1}h_5^{-1}$ for some $h_1,h_2,h_3,h_4\in H$. 
It follows that
$\sigma_1\in Hx_{k-1}H=H(x_{k-1}^{-1})^{-1}H=H(h_4^{-1}\sigma_5^{-1}h_5)^{-1}H=H\sigma_5^{-1}H$ and
$\sigma_2\in Hx_kH=H\sigma_4^{-1}H$,
which contradicts the half $H^-$-separatedness of $R$ and the equality $r\equiv st\equiv s'\sigma s''t$. 
\end{proof}

\section{A special set in $F_*$ satisfying the condition $C'(\lambda)$}\label{s:C'}

In this section we establish the condition $C'(\lambda)$ for some special subset of a free product with amalgamation.

Let $H$ be a subgroup of a group $G$. A sequence $(g_i)_{i\in n}$ in $G\setminus H$ is defined to be 
\begin{itemize}
\item {\em $H$-malnormal} if for every $i\in n$ the element $g_i$ is $H$-malnormal;
\item {\em $H$-separated} if $g_i\notin Hg_jH$ for any distinct numbers $i,j\in n$;
\item {\em $H^\pm$-separated} if $g_i\notin (Hg_jH)\cup (Hg_j^{-1}H)$ for any distinct numbers $i,j\in n$.
\end{itemize}

\begin{lemma}\label{l:C'} Let $L,H,G$ be groups such that $L\cap G=H$. Let $n\ge 3$ be a positive integer, $(a_i)_{i\in n}$ be an $H^\pm$-separated $H$-malnormal sequence in $L\setminus H$ and $(x_i)_{i\in n}$ be an $H^\pm$-separated sequence in $G\setminus H$. For every $H$-malnormal element $a\in L\setminus H$, the symmetrized hull $R^\circ$ of the set
$$R=\{a_0xa_1x\cdots a_{n-1}x:x\in G\setminus H\}\cup\{x_0ax_1a\cdots x_{n-1}a\}\subseteq L*_HG$$satisfies the condition $C'(\lambda)$ for every $\lambda>\frac{5}{2n}$.
\end{lemma}

\begin{proof} 
By Corollary~\ref{c:sym-hull}, the symmetrized hull $R^\circ$ of the set $R$ coincides with the set $\bigcup_{i=1}^8R_i$ where
$$
\begin{aligned}
R_1&=\{la_ixa_{i\oplus1}x\cdots a_{i\oplus (n-1)}xl^{-1}:i\in n,\;x\in G\setminus H,\;l\in L\setminus Ha_i^{-1} \},\\
R_2&=\{gxa_ixa_{i\oplus1}\cdots xa_{i\oplus(n-1)}g^{-1}:i\in n,\;x\in G\setminus H,\; g\in G\setminus Hx^{-1}\},\\
R_3&=\{gx_iax_{i\oplus1}a\cdots x_{i\oplus(n-1)}ag^{-1}:i\in n,\;g\in G\setminus Hx_i^{-1}\},\\
R_4&=\{lax_iax_{i\oplus1}\cdots ax_{i\oplus(n-1)}l^{-1}:i\in n,\;l\in L\setminus Ha^{-1}\},\\
R_5&=\{g xa_{i}^{-1}xa_{i\ominus1}^{-1}\cdots xa_{i\ominus(n-1)}^{-1}g^{-1}:i\in n,\;x\in G\setminus H,\;g\in G\setminus Hx^{-1} \},\\
R_6&=\{l a_{i}^{-1}xa_{i\ominus1}^{-1}x\cdots a_{i\ominus(n-1)}^{-1}xl^{-1}:i\in n,\;x\in G\setminus H,\;l\in L\setminus Ha_{i}\},\\
R_7&=\{l a^{-1}x_{i}^{-1}a^{-1}x_{i\ominus1}^{-1}\cdots a^{-1}x_{i\ominus(n-1)}^{-1}l^{-1}:i\in n,\;l\in L\setminus Ha\},\\
R_8&=\{g x_{i}^{-1}a^{-1}x_{i\ominus 1}^{-1}a^{-1}\cdots x_{i\ominus (n-1)}^{-1}a^{-1}g^{-1}:i\in n,\;g\in G\setminus Hx_{i}\},
\end{aligned}
$$
Since $\lambda>\frac5{2n}$ and each element of the set $R^\circ$ has length $2n$ or $2n+1$, the condition $C'(\lambda)$ of $R^\circ$ will follow as soon as we check that for every distinct elements $r,r'\in R^\circ$, every $\{r,r'\}$-piece $p$ has length $|p|\le 5$. To derive a contradiction, assume that some distinct elements $r,r'\in R^\circ$ have a piece of length $>5$. Then they also have a piece $p$ of length $|p|=6$. By the definition of a piece, there exist elements $u,u'\in L*_HG$ such that $r\ddoteq pu$ and $r'\ddoteq pu'$.

 Depending on the location of the elements $r,r'$ in the set $$R^\circ=R_1\cup R_2\cup R_3\cup R_4\cup R_5\cup R_6\cup R_7\cup R_8,$$
we should consider 64 cases, which reduce to 36 cases by the symmetry $\{r,r'\}=\{r',r\}$. Depending on the location of $r$ in the set $R^\circ$, we consider 8 cases.
\smallskip

1. First we assume that $r\in R_1$ and hence $r=la_ixa_{i\oplus1}xa_{i\oplus2}\cdots a_{i\oplus(n-1)}xl^{-1}$ for some $i\in n$, $x\in G\setminus H$ and $l\in L\setminus Ha_i^{-1}$. In this case the equality $r\ddoteq pu$ and Lemma~\ref{l:canrep} imply that $p=l_1g_1l_2g_2l_3g_3$ for some $l_1,l_2,l_3\in L\setminus H$ and $g_1,g_2,g_3\in G\setminus H$, and also that $$la_i=l_1h_1,\;x=h_1^{-1}g_1h_2,\; a_{i\oplus1}=h_2^{-1}l_2h_3,\; x=h_3^{-1}g_2h_4,\;\mbox{and}\;a_{i\oplus2}=h_4^{-1}l_3h_5$$ for some $h_1,h_2,h_3,h_4,h_5\in H$. 

Depending on the location of the element $r'$ in the set $R^\circ$, eight cases are possible.
\smallskip

1.1. If $r'\in R_1$, then $r'=\lambda a_jya_{j\oplus1}ya_{j\oplus2}\cdots a_{j\oplus(n-1)}y\lambda^{-1}$ for some $j\in n$, $y\in G\setminus H$, and $\lambda\in L\setminus Ha_j^{-1}$. The equality $r'\ddoteq pu'$ implies that 
$$\lambda a_j=l_1\hbar_1,\; y=\hbar_1^{-1}g_1\hbar_2,\;a_{j\oplus1}=\hbar_2^{-1}l_2\hbar_3,\;y=\hbar_3^{-1}g_2\hbar_4, \mbox{ and } a_{j\oplus2}=\hbar_4^{-1}l_3\hbar_5$$ for some $\hbar_1,\hbar_2,\hbar_3,\hbar_4,\hbar_5\in H$. The $H$-separatedness and $H$-malnormality of the sequence $(a_k)_{k\in n}$ and the equality $a_{j\oplus1}=\hbar_2^{-1}l_2\hbar_3=\hbar_2^{-1}h_2a_{i\oplus1}h_3^{-1}\hbar_3$ imply that $j\oplus1=i\oplus1$ and $\hbar_2^{-1}h_2=e=h_3^{-1}\hbar_3$. The $H$-malnormality of the element $a_{i\oplus1}=a_{j\oplus2}$ and the equality $a_{i\oplus2}=a_{j\oplus2}=\hbar_4^{-1}l_3\hbar_5=\hbar_4^{-1}h_4a_{i\oplus2}h_5^{-1}\hbar_5$ imply that $\hbar_4^{-1}h_4=e$. Then $y=\hbar_3^{-1}g_2\hbar_4=\hbar_3^{-1}h_3xh_4^{-1}\hbar_4=exe=x$ and $x=y=\hbar_1^{-1}g_1\hbar_2=\hbar_1^{-1}h_1xh_2^{-1}\hbar_2=\hbar_1^{-1}h_1xe$ implies that $\hbar_1^{-1}h_1=e$. Finally, $\lambda a_i=\lambda a_j=l_1\hbar_1=la_ih_1^{-1}\hbar_1=la_ie$ implies $\lambda=l$ and hence $r'=r$, which contradicts the choice of $r\ne r'$.
\smallskip

1.2. If $r'\in R_2$, then $r'=gya_jya_{j\oplus1}ya_{j\oplus2}\cdots ya_{j\oplus(n-1)}g^{-1}$ for some $j\in n$, $y\in G\setminus H$, and $g\in G\setminus Hy^{-1}$. In this case the equality $r'\ddoteq pu'$ does not hold by Lemma~\ref{l:canrep}.
\smallskip

1.3. If $r'\in R_3$, then $r'=gx_jax_{j\oplus1}a\cdots x_{j\oplus(n-1)}ag^{-1}$ for some $j\in n$ and $g\in G\setminus Hx_j^{-1}$. In this case the equality $r'\ddoteq pu'$ does not hold by Lemma~\ref{l:canrep}.
\smallskip

1.4. If $r'\in R_4$, then  $r'=lax_jax_{j\oplus1}\cdots ax_{j\oplus(n-1)}l^{-1}$ for some  $j\in n$ and $l\in L\setminus Ha^{-1}$. In this case the equality $r'\ddoteq pu'$ and Lemma~\ref{l:canrep} imply that 
$$la=l_1\hbar_1,\; x_{j}=\hbar_1^{-1}g_1\hbar_2,\; a=\hbar_2^{-1}l_2\hbar_3,\;x_{j\oplus1}=\hbar_3^{-1}g_2\hbar_4\;\mbox{ and }a=\hbar_4^{-1}l_3\hbar_5$$ for some $\hbar_1,\hbar_2,\hbar_3,\hbar_4,\hbar_5\in H$. The inclusion $a_{i\oplus2}=h_4^{-1}l_3h_5\in Hl_3H=HaH=Hl_2H=Ha_{i\oplus1}H$ contradicts the $H$-separatedness of the sequence $(a_k)_{k\in n}$.
\smallskip

1.5. If $r'\in R_5$, then $r'=ga_{j}^{-1}ya_{j\ominus 1}^{-1}\cdots a_{j\ominus(n-1)}^{-1}yg^{-1}$ for some $j\in n$, $y\in G\setminus H$, and $g\in G\setminus Ha_j$. In this case the equality $r'\ddoteq pu'$ does not hold by Lemma~\ref{l:canrep}. 
\smallskip

1.6. If $r'\in R_6$, then $r'=la^{-1}_{j}ya^{-1}_{j\ominus1}\cdots a_{j\ominus(n-1)}^{-1}yl^{-1}$ for some $j\in n$ and some $l\in L\setminus Ha_{j}$.  The equality $r'\ddoteq pu'$ and Lemma~\ref{l:canrep} imply that 
$$la^{-1}_{j}=l_1\hbar_1,\; y=\hbar^{-1}g_1\hbar_2,\;a_{j\ominus1}^{-1}=\hbar_2^{-1}l_2\hbar_3,\;y=\hbar_3^{-1}g_2\hbar_4,\;a_{j\ominus2}^{-1}=\hbar_4^{-1}l_3\hbar_5$$ for some $\hbar_1,\hbar_2,\hbar_3,\hbar_4,\hbar_5\in H$. The $H^\pm$-separatedness of the sequence $(a_k)_{k\in n}$ and the equalities $a_{j\ominus1}^{-1}=\hbar_2^{-1}l_2\hbar_3\in Hl_2H=Hh_2a_{i\oplus1}h_3^{-1}H=Ha_{i\oplus1}H$ and $a_{j\ominus2}^{-1}=\hbar^{-1}_4l_3\hbar_5\in Hl_3H=Hh_4a_{i\oplus2}h_5^{-1}H=Ha_{i\oplus2}H$ imply that $j\ominus 1=i\oplus1$ and $j\ominus 2=i\oplus2$,  which is a desired contradiction completing the proof of the case 1.6.
\smallskip

1.7. If $r'\in R_7$, then  $r'=la^{-1}x_{j}^{-1}a^{-1}x_{j\ominus1}^{-1}a^{-1}\cdots a^{-1}x_{j\ominus(n-1)}^{-1}l^{-1}$ for some $j\in n$ and $l\in L\setminus Ha$.  In this case the equality $r'\ddoteq pu'$ and Lemma~\ref{l:canrep} imply that $$la^{-1}=l_1\hbar_1,\; x_{j}^{-1}=\hbar_1^{-1}g_1\hbar_2,\; a^{-1}=\hbar_2^{-1}l_2\hbar_3,\; x_{j\ominus1}^{-1}=\hbar_3^{-1}g_2\hbar_4\;\mbox{ and }a^{-1}=\hbar_4^{-1}l_3\hbar_5$$ for some $\hbar_1,\hbar_2,\hbar_3,\hbar_4,\hbar_5\in H$. The inclusion $a_{i\oplus2}\in Hl_3H=Ha^{-1}H=Hl_2H=Ha_{i\oplus1}H$ contradicts the $H$-separatedness of the sequence $(a_k)_{k\in n}$. 
\smallskip

1.8. If $r'\in R_8$, then $r'=gx_{j}^{-1}a^{-1}x_{j\ominus1}^{-1}\cdots x_{j\ominus(n-1)}a^{-1}g^{-1}$ for some $j\in n$ and $g\in G\setminus Hx_{j}$. In this case the equality $r'\ddoteq pu'$ does not hold by Lemma~\ref{l:canrep}.
\smallskip

2. Next, assume that $r\in R_2$ and hence $r=gxa_ixa_{i\oplus1}\cdots xa_{i\oplus(n-1)}g^{-1}$ for some $i\in n$, $x\in G\setminus H$ and $g\in G\setminus Hx^{-1}$. In this case the equality $r\ddoteq pu$ and Lemma~\ref{l:canrep} imply that $p=g_1l_1g_2l_2g_3l_3$ for some $g_1,g_2,g_3\in G\setminus H$ and $l_1,l_2,l_3\in L\setminus H$, and also that $$gx=g_1h_1,\;a_i=h_1^{-1}l_1h_2,\; x=h_2^{-1}g_2h_3,\;\mbox{ and }\; a_{i\oplus1}=h_3^{-1}l_2h_4$$ for some $h_1,h_2,h_3,h_4\in H$. 

Depending on the location of the element $r'$ in the set $R^\circ$, we consider eight cases.
\smallskip

2.1. The case $r'\in R_1$ follows from the case 1.2.
\smallskip

2.2. If $r'\in R_2$, then $r'=\gamma ya_jya_{j\oplus1}\cdots ya_{j\oplus(n-1)}\gamma^{-1}$ for some $j\in n$, $y\in G\setminus H$ and $\gamma\in G\setminus Hy^{-1}$. The equality $r'\ddoteq pu'$ and Lemma~\ref{l:canrep} imply 
$$\gamma y=g_1\hbar_1,\;a_j=\hbar_1^{-1}l_1\hbar_2,\;y=\hbar_2^{-1}g_2\hbar_3,\;\mbox{ and }\;a_{j\oplus1}=\hbar_3^{-1}l_2\hbar_4$$
for some $\hbar_1,\hbar_2,\hbar_3,\hbar_4\in H$. The $H$-separatedness and $H$-malnormality of the sequence $(a_k)_{k\in n}$ and the equality $a_j=\hbar_1^{-1}l_1\hbar_2=\hbar_1^{-1}h_1a_ih_2^{-1}\hbar_2$ imply that $j=i$ and $\hbar_1^{-1}\hbar_1=e=h_2^{-1}\hbar_2$. The $H$-malnormality of the element $a_{i\oplus1}=a_{j\oplus1}=\hbar_3^{-1}l_2\hbar_4=\hbar_3^{-1}h_3a_{i\oplus1}h_4^{-1}\hbar_4$ implies $\hbar_3^{-1}h_3=e$. Then $y=\hbar_2^{-1}g_2\hbar_3=\hbar_2^{-1}h_2xh_3^{-1}\hbar_3=exe=x$ and hence $\gamma x=\gamma y=g_1\hbar_1=gxh_1^{-1}\hbar_1=gxe$, which implies $\gamma=g$. Then $r'=r$, which contradicts the choice of $r\ne r'$.
\smallskip

2.3. If $r'\in R_3$, then $r'=\gamma x_jax_{j\oplus 1}a\cdots x_{j\oplus(n-1)}a\gamma^{-1}$ for some $j\in n$ and $\gamma\in G\setminus Hx_j^{-1}$. In this case the equality $r'\ddoteq pu'$ and Lemma~\ref{l:canrep} imply that $$\gamma x_j=g_1\hbar_1,\;a=\hbar_1^{-1}l_1\hbar_2,\; x_{j\oplus1}=\hbar_2^{-1}g_2\hbar_3,\;\mbox{ and }\;a=\hbar_3^{-1}l_2\hbar_4$$for some $\hbar_1,\hbar_2,\hbar_3,\hbar_4\in H$. The equality $a_{i\oplus 1}=h_3^{-1}l_2h_4\in Hl_2H=H\hbar_3a\hbar_4^{-1}H=HaH=H\hbar_1^{-1}l_1\hbar_2H=Hl_1H=Hh_1a_ih_1^{-1}H=Ha_iH$ contradicts the $H$-separatedness of the sequence $(a_k)_{k\in n}$.
\smallskip

2.4. If $r'\in R_4$, then $r'=\lambda ax_jax_{j\oplus 1}\cdots ax_{j\oplus(n-1)}\lambda^{-1}$ for some $j\in n$ and $\lambda\in L\setminus Ha^{-1}$. In this case the equality $r'\ddoteq pu'$ does not hold by Lemma~\ref{l:canrep}.
\smallskip

2.5. If $r'\in R_5$, then $r'=\gamma  ya_{j}^{-1}y a_{j\ominus1}^{-1}\cdots  ya_{j\ominus(n-1)}^{-1}\gamma^{-1}$ for some $j\in n$, $y\in G\setminus H$ and $\gamma\in G\setminus Hy^{-1}$. The equality $r'\ddoteq pu'$ and Lemma~\ref{l:canrep} imply that $$\gamma y=g_1\hbar_1,\; a_{j}^{-1}=\hbar_1^{-1}l_1\hbar_2,\;y=\hbar_2^{-1}g_2\hbar_3\;\mbox{ and }\;a_{j\ominus1}^{-1}=\hbar_3^{-1}l_2\hbar_4$$ for some $\hbar_1,\hbar_2,\hbar_3,\hbar_4\in H$. The $H^\pm$-separatedness of the sequence $(a_k)_{k\in n}$ and the inclusions $a_{j}^{-1}=\hbar_1^{-1}l_1\hbar_2\in Hl_1H=Hh_1a_ih_2^{-1}H=Ha_iH$ and $a_{j\ominus1}^{-1}\in Hl_2H=Hh_3a_{i\oplus1}h_4^{-1}H=Ha_{i\oplus1}H$ imply $j=i$ and $j\ominus1=i\oplus 1$, which is a desired contradition completing the proof of the case 2.5.
\smallskip

2.6. If $r'\in R_6$, then $r'=\lambda a_{j}^{-1}ya_{j\ominus1}^{-1}\cdots a_{j\ominus(n-1)}^{-1}y\lambda^{-1}$ for some $j\in n$, $y\in G\setminus H$ and $\lambda\in L\setminus Ha_{j}$. In this case the equality $r'\ddoteq pu'$ does not hold, according to Lemma~\ref{l:canrep}.
\smallskip

2.7. If $r'\in R_7$, then the equality $r'\ddoteq pu'$ does not hold, by Lemma~\ref{l:canrep}.
\smallskip

2.8. If $r'\in R_8$, then $r'=\gamma x_{j}^{-1}a^{-1}x_{j\ominus1}^{-1}a^{-1}\cdots x_{j\ominus(n-1)}^{-1}a^{-1}\gamma^{-1}$ for some $j\in n$ and $\gamma\in G\setminus Hx_{j}$. The equality $r'\ddoteq pu'$ and Lemma~\ref{l:canrep} imply that $$\gamma x_{j}^{-1}=g_1\hbar_1,\;a^{-1}=\hbar_1^{-1}l_1\hbar_2,\; x_{j\ominus1}^{-1}=\hbar_2^{-1}g_2\hbar_3,\;\mbox{ and }a^{-1}=\hbar_3^{-1}l_2\hbar_4$$ for some $\hbar_1,\hbar_2,\hbar_3,\hbar_4\in H$. It follows that $a_{i\oplus1}=h_3^{-1}l_2h_4\in Hl_2H=H\hbar_3a^{-1}\hbar_4^{-1}H=Ha^{-1}H=H\hbar_1^{-1}l_1\hbar_2H=Hl_1H=Hh_1a_ih_2^{-1}H=Ha_iH$, which contradicts the $H$-separatedness of the sequence $(a_k)_{j\in n}$.
\smallskip

3. Assume that $r\in R_3$ and hence $r=gx_iax_{i\oplus1}ax_{i\oplus2}\cdots x_{i\oplus(n-1)}ag^{-1}$ for some $i\in n$ and $g\in G\setminus Hx_i^{-1}$. The equality $r\ddoteq pu$ implies that $g=g_1l_1g_2l_2g_3l_3$ for some $g_1,g_2,g_3\in G\setminus H$ and $l_1,l_2,l_3\in L\setminus H$, and also that $$gx_i=g_1h_1,\;a=h_1^{-1}l_1h_2,\;x_{i\oplus1}=h_2^{-1}g_2h_3,\;a=h_3^{-1}l_2h_4\;\mbox{and}\;x_{i\oplus2}=h_4^{-1}g_3h_5$$for some $h_1,h_2,h_3,h_4,h_5\in H$. 

Depending on the location of the element $r'$ in the set $R^\circ$, we consider eight cases.
\smallskip

3.1. The case $r'\in R_1$ follows from the case 1.3.
\smallskip

3.2. The case $r'\in R_2$ follows from the case 2.3.
\smallskip

3.3. If $r'\in R_3$, then $r'=\gamma x_jax_{j\oplus1}a\cdots x_{j\oplus(n-1)}a\gamma^{-1}$ for some $j\in n$ and $\gamma\in G\setminus Hx_j^{-1}$. The equality $r'\ddoteq pu'$ and Lemma~\ref{l:canrep} imply that 
$$\gamma x_j=g_1\hbar_1,\;a=\hbar_1^{-1}l_1\hbar_2,\;x_{j\oplus1}=\hbar_2^{-1}g_2\hbar_3\;\mbox{ and }\;a=\hbar_3^{-1}l_2\hbar_4$$
for some $\hbar_1,\hbar_2,\hbar_2,\hbar_4\in H$. The $H$-malnormality of $a$ and the equalities $a=\hbar_1^{-1}l_1\hbar_2=\hbar_1^{-1}h_1ah_2^{-1}\hbar_2$ and $a=\hbar_3^{-1}l_2\hbar_4=\hbar_3^{-1}h_3ah_4^{-1}\hbar_4$ imply that $e=\hbar_1^{-1}h_1=h_2^{-1}\hbar_2=\hbar_3^{-1}h_3$. The $H$-separatedness of the sequence $(x_k)_{k\in n}$ and the equality $x_{j\oplus1}=\hbar_2^{-1}g_2\hbar_3=\hbar_2^{-1}h_2x_{i\oplus1}h_3^{-1}\hbar_3=ex_{i\oplus1}e=x_{i\oplus1}$ implies $j\oplus1=i\oplus1$ and $j=i$. Also $\gamma x_i=\gamma x_j=g_1\hbar_1=gx_ih_1^{-1}\hbar_1=gx_i$ implies $\gamma=g$ and hence
$r'=r$, which contradicts the choice of $r\ne r'$.
\smallskip

3.4. If $r'\in R_4$, then the equality $r'\ddoteq pu'$ does not hold according to Lemma~\ref{l:canrep}.
\smallskip

3.5. If $r'\in R_5$, then $r'=\gamma ya_{j}^{-1}ya_{j\ominus1}^{-1}\cdots ya_{j\ominus(n-1)}^{-1}\gamma^{-1}$ for some $j\in n$, $y\in G\setminus H$ and $\gamma\in G\setminus Hy^{-1}$. The equality $r'\ddoteq pu'$ and Lemma~\ref{l:canrep} imply that $$\gamma y=g_1\hbar_1,\;a_{j}^{-1}=\hbar_1^{-1}l_1\hbar_2,\;y=\hbar_2^{-1}g_2\hbar_3\;\mbox{ and }\;a_{j\ominus 1}^{-1}=\hbar_3^{-1}l_2\hbar_4$$ for some $\hbar_1,\hbar_2,\hbar_3,\hbar_4\in H$. It follows that $a_{j\ominus1}^{-1}=\hbar_3^{-1}l_2\hbar_4\in Hl_2H=Hh_3ah_4^{-1}H=HaH=Hh_1^{-1}l_1h_2H=Hl_1H=H\hbar_1a_{j}^{-1}\hbar_2^{-1}H=Ha_{j}^{-1}H$, which contradicts the $H$-separatedness of the sequence $(a_k)_{k\in n}$.
\smallskip

3.6 and 3.7. If $r'\in R_6\cup R_7$, then the equality $r'\ddoteq pu'$ does not hold, by Lemma~\ref{l:canrep}.
\smallskip

3.8. If $r'\in R_8$, then $r'=\gamma x^{-1}_{j}a^{-1}x^{-1}_{j\ominus1}a^{-1}x^{-1}_{j\ominus2}\cdots x_{j\ominus(n-1)}^{-1}a^{-1}\gamma^{-1}$ for some $j\in n$ and $\gamma\in G\setminus Hx_j$. The equality $r'\ddoteq pu'$ and Lemma~\ref{l:canrep} imply that $$\gamma x^{-1}_{j}=g_1\hbar_1,\;a^{-1}=\hbar_1^{-1}l_1\hbar_2,\;x_{j\ominus1}^{-1}=\hbar_2^{-1}g_2\hbar_3,\;a^{-1}=\hbar_3^{-1}l_2\hbar_4\;\mbox{ and }x_{j\ominus2}^{-1}=\hbar_4^{-1}g_3\hbar_5$$ for some $\hbar_1,\hbar_2,\hbar_3,\hbar_4,\hbar_5\in H$. The $H^\pm$-separatedness of the sequence $(x_k)_{k\in n}$ and the inclusions 
$x_{j\ominus1}^{-1}\in Hg_2H=Hx_{i\oplus1}H$ and $x_{j\ominus2}^{-1}\in Hg_3H=Hx_{i\oplus2}H$ imply that $j\ominus 1=i\oplus 1$ and $j\ominus 2=i\oplus 2$ and hence $0=2$ in $n\ge 3$,  which is a contradiction completing the analysis of the case 3.8.
\smallskip

4. Assume that $r\in R_4$ and hence $r=l ax_iax_{i\oplus1}a\cdots ax_{i\oplus(n-1)}l^{-1}$ for some $i\in n$ and $l\in L\setminus Ha^{-1}$. The equality $r\ddoteq pu$ and Lemma~\ref{l:canrep} imply that $p=l_1g_1l_2g_2l_3g_3$ for some $l_1,l_2,l_3\in L\setminus H$ and $g_1,g_2,g_3\in G\setminus H$, and also that
$$l a=l_1h_1,\;x_i=h_1^{-1}g_1h_2,\;a=h_2^{-1}l_2h_3,\;x_{i\oplus 1}=h_3^{-1}g_2h_4\;\mbox{ and }\;a=h_4^{-1}l_3h_5$$for some $h_1,h_2,h_3,h_4,h_5\in H$.

Depending on the location of the element $r'$ in the set $R^\circ$, we consider eight cases.
\smallskip

4.1. The case $r'\in R_1$ follows from the case 1.4.
\smallskip

4.2. The case $r'\in R_2$ follows from the case 2.4.
\smallskip

4.3. The case $r'\in R_3$ follows from the case 3.4.
\smallskip

4.4. If $r'\in R_4$, then $r'=\lambda ax_jax_{j\oplus 1}ax_{j\oplus2}\cdots ax_{j\oplus(n-1)}\lambda^{-1}$ for some $j\in n$ and $\lambda\in L\setminus Ha^{-1}$.
The equality $r'\ddoteq pu'$ and Lemma~\ref{l:canrep} imply that $$\lambda a=l_1\hbar_1,\;x_j=\hbar_1^{-1}g_1\hbar_2,\;a=\hbar_2^{-1}l_2\hbar_3,\;x_{j\oplus1}=\hbar_3^{-1}g_2\hbar_4\;\mbox{ and }\;a=\hbar_4^{-1}l_3\hbar_5$$for some $\hbar_1,\hbar_2,\hbar_3,\hbar_4,\hbar_5\in H$. The $H$-malnormality of $a$ and the equalities $a=\hbar_2^{-1}l_2\hbar_3=\hbar_2^{-1}h_2ah_3^{-1}\hbar_3$ and $a=\hbar_4^{-1}l_3\hbar_5=\hbar_4^{-1}h_4ah_5^{-1}\hbar_5$ imply $e=\hbar_2^{-1}h_2=h_3^{-1}\hbar_3=\hbar_4^{-1}h_4$. The $H$-separatedness of the sequence $(x_k)_{k\in n}$ and the equality $x_{j\oplus1}=\hbar_3^{-1}g_2\hbar_4=\hbar_3^{-1}h_3x_{i\oplus1}h_4^{-1}\hbar_4=ex_{i\oplus1}e=x_{i\oplus1}$ imply $j\oplus1=i\oplus1$. It follows from $x_i=x_j=\hbar_1g_1\hbar_2=\hbar_1h_1^{-1}x_ih_2^{-1}\hbar_2=\hbar_1h_1^{-1}x_ie$ that $\hbar_1h_1^{-1}=e$ and finally, $\lambda a=l_1\hbar_1=l_1h_1=la$ and hence $\lambda=l$ and $r'=r$, which contradicts the choice of $r\ne r'$.
\smallskip

4.5. If $r'\in R_5$, then the equality $r'\ddoteq pu'$ does not hold, by Lemma~\ref{l:canrep}.
\smallskip

4.6. If $r'\in R_6$, then $r'=\lambda a_{j}^{-1}ya_{j\ominus1}^{-1}ya_{j\ominus2}^{-1}\cdots a_{j\ominus(n-1)}^{-1}y\lambda^{-1}$ for some $j\in n$, $y\in G\setminus H$ and $\lambda\in L\setminus Ha_{j}$. The equality $r'\ddoteq pu'$ and Lemma~\ref{l:canrep} imply that $$\lambda a_{j}^{-1}=l_1\hbar_1,\;
y=\hbar_1^{-1}g_1\hbar_2,\;a_{j\ominus1}^{-1}=\hbar_2^{-1}l_2\hbar_3\;y=\hbar_3^{-1}g_2\hbar_4\;\mbox{ and }\;a_{j\ominus2}^{-1}=\hbar_4^{-1}l_3\hbar_5$$
for some $\hbar_1,\hbar_2,\hbar_3,\hbar_4,\hbar_5\in H$. It follows that $a_{j\ominus2}^{-1}\in Hl_3H=HaH=Hl_2H=Ha_{j\ominus1}^{-1}H$, which contradicts the $H$-separatedness of the sequence $(a_k)_{k\in n}$.
\smallskip

4.7. If $r'\in R_7$, then $r'=\lambda a^{-1}x_{j}^{-1}a^{-1}x_{j\ominus1}^{-1}\cdots a^{-1}x_{j\ominus(n-1)}^{-1}\lambda^{-1}$ for some $j\in n$ and $\lambda\in L\setminus Ha$. The equality $r'\ddoteq pu'$ and Lemma~\ref{l:canrep} imply that $$\lambda a^{-1}=l_1\hbar_1,\; x_{j}^{-1}=\hbar_1^{-1}g_1\hbar_2,\;a^{-1}=\hbar_2^{-1}l_2\hbar_3\;\mbox{ and }\;x_{j\ominus1}^{-1}=\hbar_3^{-1}g_2\hbar_4 $$for some $\hbar_1,\hbar_2,\hbar_3,\hbar_4\in H$. The $H^\pm$-separatedness of the sequence $(x_k)_{k\in n}$ and the inclusions $x_{j}^{-1}\in Hg_1H=Hx_iH$ and $x_{j\ominus1}^{-1}\in Hg_2H=Hx_{i\oplus 1}H$ imply that $i=j=i\oplus2$, which is a desired contradiction completing the analysis of the case 4.7.
\smallskip

4.8. By Lemma~\ref{l:canrep}, the inclusion $r'\in R_8$ contradicts the equality $r'\ddoteq pu'\ddoteq(l_1g_1l_2g_2l_3g_3)u'$.
\smallskip

5. Assume that $r\in R_5$.  Depending on the location of the element $r'$ in the set $R^\circ$, we consider eight cases.
\smallskip

5.1--5.4. For every $k\in\{1,2,3,4\}$, the case $r'\in R_k$ follows from the case $k$.5.

5.5. The case $r'\in R_5$ follows from the case 2.2 applied to the  sequence $(a_{\ominus k}^{-1})_{k\in n}$  instead of the sequence $(a_k)_{k\in n}$. We recall that $\ominus k=0\ominus k\in n$ is the inverse element to $k$ in the cyclic group $(n,\oplus)$.

5.6. The case $r'\in R_6$ follows from the case 2.1 applied to the sequence $(a_{\ominus k}^{-1})_{k\in n}$  instead of the sequence $(a_k)_{k\in n}$.

5.7. The case $r'\in R_7$ follows from the case 2.4 applied to the $H$-malnormal element $a^{-1}$ and sequences $(a_{\ominus k}^{-1})_{k\in n}$ and $(x_{\ominus k}^{-1})_{k\in n}$ instead of the element $a$ and the sequences $(a_k)_{k\in n}$ and $(x_i)_{i\in n}$, respectively.

5.8. The case $r'\in R_8$ follows from the case 2.4 applied to the  element $a^{-1}$ and sequences $(a_{\ominus k}^{-1})_{k\in n}$ and $(x_{\ominus k}^{-1})_{k\in n}$ instead of the element $a$ and the sequences $(a_k)_{k\in n}$ and $(x_k)_{k\in n}$, respectively.
\smallskip

6. Assume that $r\in R_6$.  Depending on the location of the element $r'$ in the set $R^\circ$, we consider eight cases.
\smallskip

6.1--6.5. For every $k\in\{1,\dots,5\}$, the case $r'\in R_k$  follows from the case $k$.6.

6.6. The case $r'\in R_6$ follows from the case 1.1 applied to the sequence $(a_{\ominus k}^{-1})_{k\in n}$  instead of the sequence $(a_k)_{k\in n}$.

6.7. The case $r'\in R_7$ follows from the case 1.4 applied to the element $a^{-1}$ and sequences $(a_{\ominus k}^{-1})_{k\in n}$ and $(x_{\ominus k}^{-1})_{k\in n}$ instead of the element $a$ and the sequences $(a_k)_{k\in n}$ and $(x_k)_{k\in n}$, respectively.

6.8. The case $r'\in R_8$ follows from the case 1.2 applied to the element element $a^{-1}$ and sequences $(a_{\ominus k}^{-1})_{k\in n}$ and $(x_{\ominus k}^{-1})_{k\in17}$ instead of the element $a$ and the sequences $(a_k)_{k\in n}$ and $(x_k)_{k\in n}$, respectively.
\smallskip

7. Assume that $r\in R_7$.  Depending on the location of the element $r'$ in the set $R^\circ$, we consider eight cases.
\smallskip

7.1--7.6. For every $k\in\{1,\dots,6\}$, the case $r'\in R_k$  follows from the case $k$.7.

7.7. The case $r'\in R_7$ follows from the case 4.4 applied to the element $a^{-1}$ and sequence $(x_{\ominus k}^{-1})_{k\in n}$ instead of the element $a$ and the sequence $(x_k)_{k\in17}$.

7.8. The case $r'\in R_8$ follows from the case 1.3 applied to the element element $a^{-1}$ and sequences $(a_{\ominus k}^{-1})_{k\in n}$ and $(x_{\ominus k}^{-1})_{k\in n}$ instead of the element $a$ and the sequences $(a_k)_{k\in n}$ and $(x_k)_{i\in n}$, respectively.
\smallskip

8. Finally assume that $r\in R_8$.  Depending on the location of the element $r'$ in the set $R^\circ$, we consider eight cases.
\smallskip

8.1--8.7. For every $k\in\{1,\dots,7\}$, the case $r'\in R_k$  follows from the case $k$.8.

8.8. The case $r'\in R_8$ follows from the case 3.3 applied to the element $a^{-1}$ and sequence $(x_{\ominus k}^{-1})_{k\in n}$ instead of the element $a$ and the sequence $(x_k)_{k\in n}$.
\end{proof}

\section{Two amalgamation lemmas}\label{s:AL}

In this section we prove two amalgamation lemmas, which will be applied for performing the inductive step in the proof of Embedding Lemma~\ref{l:EL}. 

\begin{lemma}\label{l:amalgamation} Let $L,G,H$ be groups such that $L\cap G=H$ and the subgroup $H$ is malnormal in $L$.  Let $(a_k)_{k\in 18}$ be an $H^\pm$-separated sequence in $L\setminus H$ and $(x_k)_{i\in 18}$ be an $H^\pm$-separated sequence in $G\setminus H$. For every elements $a\in L\setminus H$ and $b\in H$, there exists a group $M$ such that
\begin{enumerate}
\item $G$ and $L$ are subgroups of $M$;
\item the unique homomorphism $q:L*_HG\to M$ extending the identity embedding $L\cup G\to M$ is surjective and its kernel coincides with the smallest normal subgroup of $L*_H G$ containing the set $\{a_0xa_1x\dots a_{17}x:x\in G\setminus H\}\cup\{b^{-1}x_0ax_1a\cdots x_{17}a\}$;
\item the subgroup $G$ is malnormal in $M$;
\item the sequence $(a_k)_{k\in 18}$ is $G^\pm$-separated;
\item $e=a_0xa_1x\cdots a_{17}x$ in $M$ for every $x\in G\setminus H$;
\item $e\ne lgc_1lgc_2\cdots lgc_n$ for every $l\in L\setminus H$, $g\in G\setminus H$, and $c_1,\dots,c_n\in H$. 
\end{enumerate}
\end{lemma}

\begin{proof} In the group $F_*=L*_HG$, consider the subset $$R=\{a_0xa_1x\dots a_{17}x:x\in G\setminus H\}\cup\{b^{-1}x_0ax_1a\cdots x_{17}a\}.$$ Let $(\tilde x_k)_{k\in 18}$ be the sequence defined by $$\tilde x_k=\begin{cases}b^{-1}x_0&\mbox{if $k=0$};\\
x_k&\mbox{otherwise}.
\end{cases}
$$Since $b\in H$, the $H^\pm$-separatedness of the sequence $(x_k)_{k\in 18}$ implies the $H^\pm$-separatedness of the sequence $(\tilde x_k)_{k\in 18}$. By Lemma~\ref{l:C'}, the symmetrized hull $R^\circ$ of the set $R$ satisfies the condition $C'(\tfrac17)$. Let $N$ be the smallest normal subgroup of the group $F_*$ containing the set $R$, $M=F_*/N$ be the quotient group and $q:F_*\to M$ be the quotient homomorphism. By Lemma~\ref{l:qLG}, the restriction $q{\restriction}_{L\cup G}$ is injective and $q[H]=q[L]\cap q[G]$. So, we can identify $L$ and $G$ with subgroups of $M$ such that $L\cap G=H$. After such an identification, we see that the conditions (1) and (2) of Lemma~\ref{l:amalgamation} are satisfied. So, it remains to check the conditions (3)--(5).
\smallskip

3. The $H^\pm$-separatedness of the sequences $(a_k)_{k\in 18}$ and $(x_k)_{k\in 18}$ implies the half $H^-$-separatedness of the set $R$. By Corollary~\ref{c:sym-hull}, every cyclically reduced element $r\in R^\circ$ has length $|r|=36$. Since $(1-6\frac17)|r|=\frac{36}{7}>4$, we can apply Lemma~\ref{l:malnormal} and conclude that the subgroup $G$ is malnormal in $M$.
\smallskip

4. Assuming that the sequence $(a_i)_{i\in18}$ is not $G^\pm$-separated in $M$, we can find distinct numbers $i,j\in18$ such that $a_i\in Ga_j^\pm G$, where $a_j^\pm=\{a_j,a_j^{-1}\}$. Then there exist $g,\gamma\in G$ and $\e\in\{1,-1\}$ such that $a_i=ga_j^\e\gamma$ in $M$ and hence $a_i^{-1}ga_j^\e\gamma\in N$. By Lemma~\ref{l:qLG}, every element $w\in N\setminus\{e\}$ has length $|w|>7-3=4$.  This fact implies that the element $a_i^{-1}ga_j^\e\gamma\in N$ of length $\le 4$ equals the neutral element of the group $F_*$. Then $g=a_i\gamma^{-1}a_j^{-\e}$ and $\gamma=a_j^{-\e}g^{-1}a_i$. It follows from $a_i,a_j\in L\setminus H$ and 
$1\ge |g|=|a_i\gamma^{-1}a_j^{-\e}|$ and $1\ge |\gamma|=|a_j^{-\e}g^{-1}a_i|$ that $\gamma^{-1},g^{-1}\in H$. Then $a_i=ga_j^\e\gamma\in Ha_j^{\e}H$, which contradicts the $H^\pm$-separatedness of the sequence $(a_k)_{k\in18}$ in $L$.
\smallskip

5. For every $x\in G\setminus H$, the element $a_0xa_1x\dots a_{17}x$ belongs to the set $R\subseteq N\subseteq F_*$ and hence equals $e$ in the group $M$.
\smallskip

6. To derive a contradiction, assume that $e=lgc_1lgc_2\cdots lgc_n$ for some elements $l\in L\setminus H$, $g\in G\setminus H$ and  $c_1,\dots,c_n\in H$. Then $lgc_1\cdots lgc_n\in N$ in the group $F_*$. By Lemma~\ref{l:large-piece}, there are elements $u,s,v,t\in F_*$ and a cyclically reduced element $r\in R^\circ$ such that 
$lgc_1\cdots lgc_n\equiv usv$, 
 $r\equiv st$ and $|s|>(1-3\frac17)|r|=\frac47|36|=20\tfrac17$ and hence  
$2n=|lgc_1lgc_1\cdots lgc_n|=|usv|\ge |s|\ge 21$ 
and $n\ge 11$. 

Corollary~\ref{c:sym-hull} implies that the cyclically reduced element $r$ belongs to the set $\bigcup_{i=1}^8R^\circ_i$ where 
$$
\begin{aligned}
R_1^\circ&=\{ca_ixa_{i\oplus1}x\cdots a_{i\oplus (n-1)}xc^{-1}:i\in n,\;x\in G\setminus H,\;c\in H\setminus Ha_i^{-1} \},\\
R_2^\circ&=\{cxa_ixa_{i\oplus1}\cdots xa_{i\oplus(n-1)}c^{-1}:i\in n,\;x\in G\setminus H,\; c\in H\setminus Hx^{-1}\},\\
R_3^\circ&=\{cx_iax_{i\oplus1}a\cdots x_{i\oplus(n-1)}ac^{-1}:i\in n,\;c\in H\setminus Hx_i^{-1}\},\\
R_4^\circ&=\{cax_iax_{i\oplus1}\cdots ax_{i\oplus(n-1)}c^{-1}:i\in n,\;c\in H\setminus Ha^{-1}\},\\
R_5^\circ&=\{c xa_{i}^{-1}xa_{i\ominus1}^{-1}\cdots xa_{i\ominus(n-1)}^{-1}c^{-1}:i\in n,\;x\in G\setminus H,\;c\in H\setminus Hx^{-1} \},\\
R_6^\circ&=\{c a_{i}^{-1}xa_{i\ominus1}^{-1}x\cdots a_{i\ominus(n-1)}^{-1}xc^{-1}:i\in n,\;x\in G\setminus H,\;c\in H\setminus Ha_{i}\},\\
R_7^\circ&=\{c a^{-1}x_{i}^{-1}a^{-1}x_{i\ominus1}^{-1}\cdots a^{-1}x_{i\ominus(n-1)}^{-1}c^{-1}:i\in n,\;c\in H\setminus Ha\},\\
R_8^\circ&=\{c x_{i}^{-1}a^{-1}x_{i\ominus 1}^{-1}a^{-1}\cdots x_{i\ominus (n-1)}^{-1}a^{-1}c^{-1}:i\in n,\;c\in H\setminus Hx_{i}\},
\end{aligned}
$$
Depending on the location of $r$ is the set $\bigcup_{i=1}^8R_i^\circ$, we consider  separately four cases.
\smallskip

1. If $r\equiv st\in R_1^\circ\cup R^\circ_2$, then $s$ can be written as $s\equiv s'\sigma s''$ for some $s',s''\in F_*$ such that $1\le |s'|\le 2$ and $\sigma=a_ixa_{i\oplus 1}$ for some $i\in 18$ and $x\in G\setminus H$. The equalities 
$$l(gc_1)l(gc_2)\cdots l(gc_n)\equiv usv\equiv us'\sigma s''v,$$  $\sigma=a_ixa_{i\oplus 1}$ and Lemma~\ref{l:canrep} imply that the length $|us'|$ is even and also that $$a_i=h_0^{-1}lh_1,\;x=h_1^{-1}gc_jh_2,\;\mbox{ and }\; a_{i\oplus1}=h_2^{-1}lh_3$$
for some $j\in\{2,\dots,n-1\}$ and some $h_1,h_2,h_3$. It follows that $a_{i\oplus1}\in HlH=Ha_iH$, which contradicts the $H$-separatedness of the sequence $(a_k)_{k\in 18}$.
\smallskip

2. If $r\equiv st\in R_3^\circ\cup R^\circ_4$, then $s$ can be written as $s\equiv s'\sigma s''$ for some $s',s''\in F_*$ such that $1\le |s'|\le 2$ and $\sigma=x_iax_{i\oplus1}$ for some $i\in 18$. The equalities 
$$l(gc_1)l(gc_2)\cdots l(gc_n)\equiv usv\equiv us'\sigma s''v,$$  $\sigma=x_iax_{i\oplus1}$ and Lemma~\ref{l:canrep} imply that the length $|us'|$ is odd and also that 
$$x_i=h_0^{-1}gc_jh_1,\;a=h_1^{-1}lh_2,\;x_{i\oplus1}=h_2^{-1}gc_{j+1}h_3$$
for some $j\in \{1,\cdots n-2\}$ and some $h_1,h_2,h_3\in H$. It follows that $x_{i\oplus1}\in HgH=Hx_iH$, which contradicts the $H$-separatedness of the sequence $(x_k)_{k\in 18}$.
\smallskip

3. If $r\equiv st\in R_5^\circ\cup R^\circ_6$, then $s$ can be written as $s\equiv s'\sigma s''$ for some $s',s''\in F_*$ such that $1\le |s'|\le 2$ and $\sigma=a^{-1}_ixa_{i\ominus 1}$ for some $i\in 18$ and $x\in G\setminus H$. The equalities 
$$l(gc_1)l(gc_2)\cdots l(gc_n)\equiv usv\equiv us'\sigma s''v,$$  $\sigma=a_i^{-1}xa^{-1}_{i\ominus 1}$ and Lemma~\ref{l:canrep} imply that the length $|us'|$ is even and also that $$a_i^{-1}=h_0^{-1}lh_1,\;x=h_1^{-1}gc_jh_2,\;\mbox{ and }\; a^{-1}_{i\ominus1}=h_2^{-1}lh_3$$
for some $j\in\{2,\dots,n-1\}$ and some $h_1,h_2,h_3$. It follows that $a^{-1}_{i\ominus1}\in HlH=Ha^{-1}_iH$, which contradicts the $H$-separatedness of the sequence $(a_k)_{k\in 18}$.
\smallskip

4. If $r\equiv st\in R_7^\circ\cup R^\circ_8$, then $s$ can be written as $s\equiv s'\sigma s''$ for some $s',s''\in F_*$ such that $1\le |s'|\le 2$ and $\sigma=x_i^{-1}ax^{-1}_{i\ominus1}$ for some $i\in 18$. The equalities 
$$l(gc_1)l(gc_2)\cdots l(gc_n)\equiv usv\equiv us'\sigma s''v,$$  $\sigma=x_iax_{i\oplus1}$ and Lemma~\ref{l:canrep} imply that the length $|us'|$ is odd and also that 
$$x_i^{-1}=h_0^{-1}gc_jh_1,\;a=h_1^{-1}lh_2,\;x^{-1}_{i\ominus1}=h_2^{-1}gc_{j+1}h_3$$
for some $j\in \{1,\cdots n-2\}$ and some $h_1,h_2,h_3\in H$. It follows that $x_{i\ominus1}\in HgH=Hx^{-1}_iH$, which contradicts the $H$-separatedness of the sequence $(x_k)_{k\in 18}$.
\end{proof}

By analogy we can prove that following modification of Lemma~\ref{l:amalgamation}.

\begin{lemma}\label{l:amalgamation2} Let $L,G,H$ be groups such that $L\cap G=H$ and the subgroup $H$ is malnormal in $L$. For every $H^\pm$-separated sequence $(a_i)_{i\in 18}$ in $L\setminus H$ there exists a group $M$ such that
\begin{enumerate}
\item $G$ and $L$ are subgroups of $M$;
\item the unique homomorphism $q:L*_HG\to M$ extending the identity embedding $L\cup G\to M$ is surjective and its kernel coincides with the smallest normal subgroup of $L*_H G$ containing the set $\{a_0xa_1x\dots a_{17}x:x\in G\setminus H\}$;
\item the subgroup $G$ is malnormal in $M$;
\item the sequence $(a_k)_{k\in 18}$ is $G^\pm$-separated;
\item $e=a_0xa_1x\cdots a_{17}x$ in $M$ for every $x\in G\setminus H$;
\item $e\ne lgc_1lgc_2\cdots lgc_n$ for every $l\in L\setminus H$, $g\in G\setminus H$, and $c_1,\dots,c_n\in H$. 
\end{enumerate}
\end{lemma}



\section{An embedding lemma}\label{l:EL}

In this section we prove Embedding Lemma~\ref{l:EL}, which will be used as an iductive step in the construction of the $36$-Shelah group in Theorem~\ref{t:shelah}. In the proof of this lemma we shall apply Lemma~\ref{l:dim} on the existence of $H^\pm$-separated sequences in sets of large dimension.

For a subset $A$ of a semigroup $X$ we denote by $\dim(A;X)$  the smallest cardinality of a subset $B\subseteq S$ such that every element $a\in A$ can be written as $a=b_1\cdots b_n$ for some $b_1,\dots,b_n\in B$. It is clear that $\dim(A;X)\le |A|$, and $\dim(A;X)=|A|$ if $A$ is uncountable. We shall write $\dim(X)$ instead of $\dim(X;X)$.
 
\begin{lemma}\label{l:dim} Let $H$ be a subgroup of a group $X$, $A$ be a subset of $X$, and $n\in\IN$. If $\dim(A;X)+2n-2<\dim(A;X)$, then there exists an $H^\pm$-separated sequence $(x_i)_{i\in n}$ in $A\setminus H$.
\end{lemma}

\begin{proof}  By the definition of $\dim(B;X)$, there exists a set $B\subseteq X$  of cardinality $|B|=\dim(H;X)$ such that $H$ is contained in the subsemigroup of $X$, generated by the set $B$. For every $k\in n$ choose inductively an element $x_k\in A\setminus B_k$ where $B_k$ is the subsemigroup of $X$, generated by the set $B\cup\{x_i\}_{i\in k}\cup\{x_i^{-1}\}_{i\in k}$. Since $\dim(B_k)\le |B|+2k=\dim(H)+2k\le \dim(H)+2(n-1)<\dim(A)$, the element $x_k$ does exist. It follows from $H\cup \bigcup_{i\in k}(Hx_iH\cup Hx_i^{-1}H)\subseteq B_k$ that $x_k\notin H\cup \bigcup_{i\in k}(Hx_iH\cup Hx_i^{-1}H)$, which implies that the sequence $(x_i)_{i\in n}$ is $H^\pm$-separated.
\end{proof}

\begin{lemma}\label{l:EL} Let $G$ be an infinite group and $\mathcal S\subseteq\{S\subseteq G:\dim(S;G)=|G|\}$ be a set of cardinality $1\le |\mathcal S|\le|G|$. Then $G$ is a malnormal subgroup of a group $L$ containing an element $\lambda\in L\setminus G$ of order $51$ such that  
\begin{enumerate}
\item the group $L$ is generated by the set $\{\lambda\}\cup G$;
\item $\lambda g\lambda^2\lambda^3g\cdots \lambda^{18}g=e$ for every $g\in G$;
\item for every $c\in G$, $a\in L\setminus G$ and $S\in \mathcal S$ there are elements $x_0,\dots,x_{17}\in S$ such that $c=x_0ax_1ax_2a\cdots x_{17}a$;
\item  for any subset $C\subseteq G$  of cardinality  $|C|<\mathrm{cf}(|G|)$ and every element $l\in L\setminus G$, there exists an element $g\in G$ such that $e\ne lgc_1lgc_2\cdots lgc_n$ for all  $c_1,c_1,\dots,c_n\in C$;
\item if $|G|=\w$, then any subset $A\subseteq G$ with $\dim(A;L)<\w$ has $\dim(A;G)<\w$.
\end{enumerate} 
\end{lemma}

\begin{proof} Let  $\kappa$ be the cardinality of the group $G$. Let $L_0$ be a cyclic group of order $|L_0|=51$ such that  $L_0\cap G=\{e\}$ where $e=ee$ is the identity element of the groups $L_0$ and $G$. Let $\lambda$ be the generator of the cyclic group $L_0$. Since $\lambda$ has order 51, the sequence $(\lambda^{i+1})_{i\in 18}$ is $\{e\}^\pm$-separated in $L_0$.

Lemma~\ref{l:canrep} implies that every element $x$ of the free product $L_0*G=L_0*_{\{e\}}G$ has a unique canonical representation $x_1\cdots x_n$, which allows us to define the set $$\supp(x)\defeq\{x_1,\dots,x_n\}\subseteq L_0\cup G$$ called the {\em support} of $x$.

Let $\{(c_\alpha,a_\alpha,S_\alpha)\}_{\alpha\in\kappa}$ be an enumeration of the set $G\times (L_0*G)\times  \mathcal S$ such that $c_0=e=a_0$. By induction we shall define increasing transfinite sequences $(H_\alpha)_{\alpha\in \kappa}$, and $(G_{\alpha})_{\alpha\in\kappa}$  of subgroups of $G$ and a transfinite sequence $(x_{\alpha})_{\alpha\in\kappa}$ of functions $x_\alpha:18\to G$ such that for every $\alpha\in\kappa$ the following conditions are satisfied:
\begin{enumerate}
\item[(a)] the subgroup $H_\alpha$ is generated by the set $\{c_\alpha\}\cup(\supp(a_\alpha)\cap G)\cup G_{<\alpha}$ where $G_{<\alpha}=\bigcup_{\beta<\alpha}G_\beta$;
\item[(b)] $\{x_\alpha(i)\})_{i\in 18}\subseteq S_\alpha\setminus H_{\alpha}$ and the sequence $(x_\alpha(i))_{i\in 18}$ is $H_\alpha^\pm$-separated;
\item[(c)] $G_{\alpha}$ is generated by the set $H_{\alpha}\cup\{x_\alpha(i)\}_{i\in 18}$.
\end{enumerate}
To start the inductive construction, let $H_0=\{e\}$ and observe that $H_0$ is generated by the set $\{c_0\}\cup(\supp(a_0)\cap K)\cap G_{<0}=\{e\}\cup (\supp(e)\cap K)\cup \emptyset=\{e\}$ and hence $\dim(H_0;G)=1<\dim(S_0;G)=|G|$. By Lemma~\ref{l:dim}, there exists a function $x_\alpha:18\to G\setminus H_0$ such that the sequence $(x_\alpha(i))_{i\in 18}$ is $H_0^\pm$-separated. Let $G_0$ be the subgroup of $G$, generated by the set $\{x_0(i)\}_{i\in 18}$. 

Now assume that for some ordinal $\alpha\in\kappa$ we have constructed increasing  sequences of subgroups $(H_\beta)_{\beta<\alpha}$ and $(G_\alpha)_{\beta<\alpha}$ and a sequence of  functions $(x_\beta:18\to G)_{\beta<\alpha}$  satisfying the inductive conditions (a)--(c). Let $G_{<\alpha}=\bigcup_{\beta<\alpha}G_\beta$ and define the subgroup $H_\alpha$ by the condition (a). The inductive conditions (a)--(c) imply that $\dim(H_\alpha;G)+34<|G|=\dim(S_\alpha;G)$. By Lemma~\ref{l:dim}, there exists a function $x_\alpha:18\to S_\alpha\setminus H_\alpha$ such that the sequence $(x_i)_{i\in 18}$ is $H_\alpha^\pm$-separated. Finally, let $G_\alpha$ be the subgroup of $G$, generated by the set $H_\alpha\cup\{x_\alpha(i)\}_{i\in 18}$. This completes the inductive step and also the inductive construction of the sequences $(H_\alpha)_{\alpha\in\kappa}$, $(G_\alpha)_{\alpha\in\kappa}$, and $(x_\alpha)_{\alpha\in\kappa}$. The inductive condition (a) implies that 
$$\bigcup_{\alpha\in\kappa}H_\alpha=\bigcup_{\alpha\in\kappa}G_\alpha=\{c_\alpha\}_{\alpha\in\kappa}=G.$$
\smallskip

Next, we shall inductively construct two increasing transifinite sequences of groups $(L_\alpha)_{\alpha\in\kappa}$ and $(M_\alpha)_{\alpha\in\kappa}$, a transfinite sequence of elements $(\check a_{\alpha})_{\alpha\in\kappa}\in\prod_{\alpha\in\kappa}L_{\alpha}$ and two transfinite sequences of sets $(R_\alpha)_{\alpha\in\kappa}$  and $(Q_\alpha)_{\alpha\in \kappa}$ such that for every $\alpha\in\kappa$ the following conditions are satisfied:
\begin{itemize}
\item[(d)] $M_{<\alpha}\cup H_\alpha\subseteq L_\alpha$ where $M_{<\alpha}=\bigcup_{\beta<\alpha}M_\beta$;
\item[(e)] $L_\alpha\cap G=H_\alpha$ and the subgroup $H_\alpha$ is malnormal in $L_\alpha$;
\item[(f)] $R_\alpha=\{\lambda x\lambda^2x\lambda^3x\cdots \lambda^{18}x:x\in H_\alpha\setminus G_{<\alpha}\}\subseteq M_{<\alpha}*_{G_{<\alpha}}H_\alpha$;
\item[(g)] the unique homomorphism $M_{<\alpha}*_{G_{<\alpha}}H_\alpha\to L_\alpha$ that extends the identity inclusion $M_{<\alpha}\cup H_\alpha\to L_\alpha$ is surjective and its kernel coincides with the smallest normal subgroup of $M_{<\alpha}*_{G_{<\alpha}}H_\alpha$ that contains the set $R_\alpha$;
\item[(h)] the group $L_\alpha$ is generated by its subset $\{\lambda\}\cup H_\alpha$;
\item[(i)] the sequence $(\lambda^{i+1})_{i\in 18}$ is $H_\alpha^\pm$-separated in $L_\alpha$;
\item[(j)] $\check a_\alpha=q_{\alpha}(a_\alpha)\in L_\alpha$ where $q_{\alpha}:L_0*H_{\alpha}\to L_{\alpha}$ is a unique homomorphism extending the identity inclusion $L_0\cup H_{\alpha}\to L_{\alpha}$;
\item[(k)] if $\check a_\alpha\in H_{\alpha}$, then $Q_{\alpha}=\{\lambda x\lambda ^2x\dots \lambda^{18}x:x\in G_\alpha\setminus H_\alpha\}\subseteq L_\alpha*_{H_\alpha}G_\alpha$;
\item[(l)] if $\check a_\alpha\notin H_{\alpha}$, then  $Q_{\alpha}=\{\lambda x\lambda^2x\dots \lambda^{18}x:x\in G_\alpha\setminus H_\alpha\}\cup\{c_\alpha^{-1} x_\alpha(0)\check a_\alpha x_\alpha(1)\check a_\alpha\cdots x_\alpha(17)\check a_\alpha\}$;
\item[(m)] $L_\alpha\cup G_\alpha\subseteq M_\alpha$ and the unique homomorphism $L_{\alpha}*_{H_{\alpha}} G_{\alpha}\to M_{\alpha}$ extending the identity embedding $L_\alpha\cup G_\alpha\to M_\alpha$ is surjective and its kernel coincides with the smallest normal subgroup of $L_{\alpha}*_{H_{\alpha}} G_{\alpha}$ containing the set $Q_{\alpha}$;
\item[(n)] $M_\alpha\cap G=G_\alpha$ and the subgroup $G_\alpha$ is malnormal in $M_\alpha$;
\item[(o)] the group $M_\alpha$ is generated by its subset $\{\lambda\}\cup G_\alpha$;
\item[(p)] the sequence $(\lambda^{i+1})_{i\in 18}$ is $G_\alpha^\pm$-separated in $M_\alpha$;
\item[(r)] for every  $l\in L_\alpha\setminus H_\alpha$ and $g\in G_\alpha\setminus H_\alpha$ we have $e\ne lg h_1lg h_2\cdots lg h_n$ in $M_\alpha$ for arbitrary elements $h_1,h_2,\dots,h_n\in H_\alpha$.
\end{itemize}

To start the inductive construction, put $R_0=N_0=\{e\}$, $\check a_0=e$, and let $M_0=L_0$.  Then $M_0\cap K=\{e\}=H_0$ and the subgroup $H_0=\{e\}$ is malnormal in $M_0$. Assume that for some positive ordinal $\alpha$ we have constructed increasing sequences of groups $(L_\beta)_{\beta<\alpha}$ and $(M_\beta)_{\beta<\alpha}$ satisfying the inductive conditions (d)--(r). Consider the groups $M_{<\alpha}=\bigcup_{\beta<\alpha}M_\beta$ and $G_{<\alpha}=\bigcup_{\beta<\alpha}G_\beta$. The inductive condition (o) implies that the group $M_{<\alpha}$ is generated by the set $L_0\cup G_{<\alpha}$. The inductive condition (n) implies that $M_{<\alpha}\cap K=G_{<\alpha}$ and $G_{<\alpha}$ is a malnormal subgroup of $M_{<\alpha}$. Let $R_\alpha$ be the subset of $M_{<\alpha}*_{G_{<\alpha}}H_\alpha$, defined in the inductive condition (f). The inductive condition (p) implies that the sequence $(\lambda^{i+1})_{i\in 18}$ is $G_{<\alpha}^\pm$-separated in the group $M_{<\alpha}$.  By Lemma~\ref{l:amalgamation2}, there exists a group $L_\alpha$ satisfying the inductive conditions (d), (e), (g), (i). Since the group $M_{<\alpha}$ is generated by the set $L_0\cup G_{<\alpha}$, the inductive condition  (g) implies that the group $L_\alpha$ is generated by the set $L_0\cup H_\alpha$ and hence satisfies the inductive condition (h).
\smallskip

Let $q_\alpha:L_0*H_\alpha\to L_\alpha$ be a unique homomorphism extending the identity inclusion $L_0\cup H_\alpha\to L_\alpha$. Let $\check a_\alpha=q_\alpha(a_\alpha)$. Consider the subset $Q_\alpha$ of $L_\alpha*_{H_\alpha}G_\alpha$, defined by the inductive conditions (k), (l). By the inductive condition (i), the sequence $(\lambda^{i+1})_{i\in 18}$ is $H_\alpha^\pm$-separated in $L_\alpha\setminus H_\alpha$ and by the choice of the function $x_\alpha:18\to S_\alpha \setminus H_\alpha$, the sequence $(x_\alpha(i))_{i\in 18}$ is $H_\alpha^\pm$-separated in $G_\alpha\setminus H_\alpha$.  Lemmas~\ref{l:amalgamation} and \ref{l:amalgamation2} yield a group $M_\alpha$ satisfying the inductive conditions (m), (n), (p), (r). The inductive conditions (h) and (m) imply that the group $M_\alpha$ satisfies the inductive condition (o).
 This completes the inductive step.
\smallskip

After completing the inductive construction, consider the group $L=\bigcup_{\alpha\in\kappa}L_\alpha=\bigcup_{\alpha<\kappa}M_\alpha$. The inductive conditions (a) and (d) ensure that $G=\bigcup_{\alpha\in\kappa}H_\alpha\subseteq \bigcup_{\alpha\in \kappa}L_\alpha=L$, and the inductive condition (e) implies that the subgroup $G$ is malnormal in $L$. It remains to check that the group $L$ and the generator $\lambda$ of the group $L_0\subseteq L$ satisfy the conditions (1)--(5) of Lemma~\ref{l:EL}.
\smallskip

1. The inductive conditions (a) and (h)  imply that the group $L$ is generated by the set $\{\lambda\}\cup G$. 
\smallskip

2. Let $g\in G$ be any element. If $g=e$, then $\lambda g\lambda^2g\cdots \lambda^{18}g=\lambda^{153}=e$ as $\lambda$ has order 51.
If $g\ne e$, then there exists an ordinal $\alpha\in\kappa$ such that $g\in H_\alpha\setminus G_{<\alpha}$ or $g\in G_{\alpha}\setminus H_\alpha$. In this case the equality $\lambda g\lambda^2g\cdots\lambda^{18}g=e$ follows from the inductive conditions (f), (g) or (k), (l), (m).
\smallskip

3. To check the third condition of Lemma~\ref{l:EL}, take any $c\in G$, $a\in L\setminus G$ and $S\in\mathcal S$. Let $q:L_0*G\to L$ be the unique homomorphism extending the identity inclusion $L_0\cup G\to L$. Since the group $L$ is generated by the set $L_0\cup G$, there exists an element $a'\in L_0*G$ such that $a=q(a')$. Since $(c,a',S)\in G\times(L_0*G)\times\mathcal S=\{(c_\alpha,a_\alpha,S_\alpha):\alpha\in\kappa\}$, there exists $\alpha\in\kappa$ such that $(c_\alpha,a_\alpha,S_\alpha)=(c,a',S)$. Let $(x_i)_{i\in 18}=(x_\alpha(i))_{i\in 18}$.  Observe that $\check a_\alpha=q_\alpha(a_\alpha)=q(a')=a\in L\setminus G\subseteq L\setminus H_\alpha$ and hence 
$$c^{-1}x_0ax_1a\cdots x_{17}a=c_\alpha^{-1}x_\alpha(0)\check a_\alpha x_\alpha(1)\check a_\alpha\cdots x_{\alpha}(17)\check a_\alpha\in Q_\alpha$$and 
$c=x_0ax_1\cdots x_{17}a$ in $M_\alpha\subseteq L$ by the inductive condition (m).
\smallskip

4. Given any subset $C\subseteq G$ of cardinality $|C|<\mathrm{cf}(\kappa)$ and any $l\in L\setminus G$, find an ordinal  $\alpha\in\kappa$ such that $C\subseteq H_\alpha$ and $l\in L_\alpha\setminus G=L_\alpha\setminus H_\alpha$. Choose any element $g\in G_\alpha\setminus H_\alpha$. It exists because $\{x_\alpha(i)\}_{i\in 18}\subseteq G_\alpha\setminus H_\alpha$. By the inductive condition (r), for every elements $h_1,\dots,h_n\in C\subseteq H_\alpha$ we have $e\ne lgh_1\cdots lgh_n$ in $M_\alpha\subseteq L$.
\smallskip

5. Assume that $|G|=\w$ and $A\subseteq G$ is a set such that $\dim(A;L)<\w$. Then $A$ is contained in a subsemigroup $S$ of $L$, generated by some finite subset $F$. Since $L=\bigcup_{\alpha\in\kappa}L_\alpha$, there exists an ordinal $\alpha\in\kappa$ such that $F\subseteq L_\alpha$ and hence $A\subseteq S\subseteq L_\alpha$. Then $A\subseteq G\cap L_\alpha=H_\alpha$, by the inductive condition (e). Since $\kappa=|G|=\w$, the inductive conditions (a) and (c) ensure that the subgroup $H_\alpha$ of $G$ is finitely generated and hence $\dim(A;G)\le\dim(H_\alpha;G)<\w$.
\end{proof} 

We shall also need the following embedding result.

\begin{lemma}\label{l:emb0} Let $\kappa$ be an infinite cardinal. Every group $G$ of cardinality $|G|\le\kappa$ is isomorphic to a malnormal subgroup of a group $M$ such that $\dim(M)=|M|=\kappa$.
\end{lemma}

\begin{proof} Take any group $L$ such that $L\cap G=\{e\}$ and $\dim(L)=|L|=\kappa$. Consider the free product $M=L*_{\{e\}}G$. Applying Lemma~\ref{l:canrep}, one can show that $G$ is a malnormal subgroup of $M$. It is clear that $|M|=\max\{|L|,|G|\}=\kappa$. 

It remains to prove that $\dim(M)=\kappa$. This equality is trivially true if the cardinal $\kappa$ is uncountable.

Next, consider the (more complicated) case of $\kappa=\w$. Assuming that $\dim(M)<\kappa=\w$, we can find a finite set $F\subseteq M$ generating the group $M$. By Lemma~\ref{l:canrep}, every element $x\in M$ has a unique canonical representation $x_1\cdots x_n$, which determines the {\em support} $\supp(x)=\{x_1,\cdots x_n\}$ of $x$. 

We claim that the finite set $E=\bigcup_{x\in F}(\supp(x)\cap L)$ generates the group $L$. Indeed, for any element $l\in L$, there exist elements $a_1,\dots,a_n\in F$ such that $l=a_1\cdots a_n$. Analyzing possible cancellations in $a_1\cdots a_n$ and applying Lemma~\ref{l:canrep}, we can prove that $l$ belongs to the subsemigroup of $L$, generated by the set $L\cap\bigcup_{i=1}^n\supp(a_i)\subseteq E$. Therefore, the group $L$ is finitely generated, which contradicts the choice of $L$. This contradiction shows that $\dim(M)=\w=\kappa$.
\end{proof}

\section{The main result}\label{s:main}

In this section we prove our main result, which is a more elaborated version of Theorem~\ref{t:main} announced in the Introduction.

\begin{theorem}\label{t:shelah} Let $\kappa$ be an infinite cardinal such that $\kappa^+=2^\kappa$. Every group $H$ of cardinality $|H|\le\kappa$ is isomorphic to a malnormal subgroup of a group $G$  such that 
\begin{enumerate}
\item $|G|=\kappa^+$;
\item $G$ is $36$-Shelah;
\item for every subset $X\subseteq G$ of cardinality $|X|<|G|$ there exists an element $g\in G$ such that $X\cap gXg^{-1}\subseteq\{e\}$;
\item the group $G$ is simple;
\item for every subset $X\subseteq G$ of cardinality $|X|<|G|$ there exists an element $a\in G$ of order $51$ such that $axa^2xa^3x\cdots a^{18}x=e$ and hence $\cov(X;\A_G)\le1$.
\item for every subset $C\subseteq G$ of cardinality $|C|<\mathrm{cf}(\kappa)$ there exists an element $x\in G$ such that $e\ne xc_1xc_2\cdots xc_n$ for all $c_1,\dots, c_n\in C$;
\item $G$ is projectively $\mathsf{T_{\!1}S}$-discrete;
\item $G$ is absolutely $\mathsf{T_{\!1}S}$-closed;
\item $\mathrm{cf}(\kappa)\le\cov(G;\A_G)\le\kappa$.
\end{enumerate}
\end{theorem}

\begin{proof} Fix a group $H$ of cardinality $|H|\le\kappa$. By Lemma~\ref{l:emb0}, $H$ is isomorphic to a malnormal subgroup of some group $G_0$ with $\dim(G_0)=|G_0|=\kappa$. We lose no generality assuming that the underlying set of the group $G_0$ coincides with the cardinal $\kappa$.

Let $\{S_\alpha\}_{\alpha\in 2^\kappa}$ be an enumeration of all subsets of cardinality $\kappa$ in the cardinal $\kappa^+$.   Applying Lemma~\ref{l:EL}, construct inductively an increasing transfinite sequences of groups $(G_\alpha)_{\alpha\in\kappa^+}$ starting from the group $G_0$ such that for every ordinal $\alpha\in\kappa^+$ the following conditions are satisfied:
\begin{enumerate}
\item[(i)] the underlying set of the group $G_\alpha$ coincides with the ordinal $\kappa\cdot(1+\alpha)\subset \kappa^+$;
\item[(ii)] $G_\alpha$ is a malnormal subgroup of $G_{\alpha+1}$;
\item[(iii)] there exists an element $\lambda_\alpha\in G_{\alpha+1}\setminus G_\alpha$ of order 51 such that the group $G_{\alpha+1}$ is generated by the set $\{\lambda_\alpha\}\cup G_\alpha$ and  $\lambda x\lambda^2x\lambda^3x\cdots \lambda^{18}x=e$ for all $x\in G_\alpha$;
\item[(iv)] for every $c\in G$, $a\in G_{\alpha+1}\setminus G_\alpha$ and $S\in\{S_\gamma:\gamma\le\alpha,\;S_\gamma\subseteq G_\alpha,\;\dim(S_\gamma;G_\alpha)=\kappa\}$ there are elements $x_1,\dots,x_{18}\in S$ such that $c=x_1ax_2ax_3a\cdots x_{18}a$;
\item[(v)]  for any subset $C\subset G_\alpha$  of cardinality  $|C|<\mathrm{cf}(\kappa)$ there exists an element $x\in G_{\alpha+1}$ such that $e\ne xc_1xc_2\cdots xc_n$ for all  $c_1,c_1,\dots,c_n\in C$;
\item[(vi)] if $\kappa=\w$, then any subset $A\subseteq G_\alpha$ with $\dim(A;G_{\alpha+1})<\w$ has $\dim(A;G_\alpha)<\w$.
\end{enumerate}
We claim that the group $G=\bigcup_{\alpha\in\kappa^+}G_\alpha$ has the properties claimed in Theorem~\ref{t:shelah}. This will be proved in the following  lemmas. 
\smallskip

\begin{lemma} The group $H$ is isomorphic to a malnormal subgroup of $G$.
\end{lemma}

\begin{proof} The group $H$ is isomorphic to a malnormal subgroup of the group $G_0$, by the choice of $G_0$. It will be convenient to identify $H$ with this malnormal subgroup of $G_0$. To prove that $H$ is malnormal in $G$, take any element $x\in G\setminus H$. If $x\in G_0$, then $H\cap xHx^{-1}=\{e\}$ by the malnormality of the subgroup $H$ in $G_0$. So, we assume that $x\in G\setminus G_0$. The inductive condition (i) ensures that  $G\setminus G_0=\bigcup_{\alpha\in\kappa^+}G_{\alpha+1}\setminus G_\alpha$ and hence $x\in G_{\alpha+1}\setminus G_\alpha$ for some $\alpha\in\kappa^+$. The inductive condition (ii) ensures that $G_\alpha$ is a malnormal subgroup in $G_{\alpha+1}$ and hence $H\cap xHx^{-1}\subseteq G_\alpha\cap xG_\alpha x^{-1}\subseteq\{e\}$, witnessing that the subgroup $H$ is malnormal in $G$.
\end{proof} 

The inductive condition (i) implies that the group $G$ satisfies the condition (1) of Theorem~\ref{t:shelah}:

\begin{lemma} The group $G$ has cardinality $|G|=\kappa^+=2^\kappa$.
\end{lemma}

\begin{lemma}\label{l:36} The group $G$ is $36$-Shelah.
\end{lemma}

\begin{proof} Take any subset $S\subseteq G$ of cardinality $|S|=|G|=\kappa^+$ and any element $c\in G$. 

\begin{claim}\label{cl:delta} There exists an ordinal $\delta\in\kappa^+$ such that $\dim(S\cap G_\delta;G_\delta)=\kappa$.
\end{claim}

\begin{proof} If $\kappa$ is uncountable, then by the regularity of the cardinal $\kappa^+=|S|$, there exists an ordinal $\delta\in\kappa^+$ such that $|S\cap G_\delta|=\kappa$. Since $\kappa$ is uncountable, $\dim(S\cap G_\delta;G_\delta)=|S\cap G_\delta|=\kappa$. 

The case of $\kappa=\w$ is more complicated. To derive a contradiction, assume that for every $\alpha\in\kappa^+=\w_1$ the cardinal $\dim(S\cap G_\alpha;G_\alpha)$ is finite and hence there exists a finite set $F\subseteq G_\alpha$ whose group hull contains the set $S\cap G_\alpha$, Let $\xi(\alpha)\le\alpha$ be the smallest ordinal such that the group $G_{\xi(\alpha)}$ contains a finite subset $F_\alpha$ whose group hull contains the set $S\cap G_\alpha$. If the ordinal $\alpha$ is limit, then the inductive condition (i) guarantees that $G_\alpha=\bigcup_{\beta\in\alpha}G_\beta$ and hence $\xi(\alpha)<\alpha$. By Fodor's Theorem 8.7 in  \cite{Jech}, there exists an ordinal $\gamma\in \kappa^+=\w_1$ and an uncountable set $\Omega\subseteq\w_1$ of limits ordinals such that $\xi(\alpha)=\gamma$ and hence $F_\alpha\subseteq G_\gamma$ for all $\alpha\in\Omega$.
 Since the group $G_\gamma$ is countable, by the Pigeonhole Principle, there exists a finite set $F\subseteq G_\gamma$ such that the set $\Lambda=\{\alpha\in\Omega:F_\alpha=F\}$ is uncountable. Then for every $\alpha\in\Lambda$ the set $S\cap G_\alpha$ is contained in the countable subsemigroup $\langle F_\alpha\rangle=\langle F\rangle$ generated by the finite set $F$ and hence the set $S=\bigcup_{\alpha\in\Lambda}(S\cap G_\alpha)\subseteq\langle F\rangle$ is countable, which contradicts the choice of $S$. This contradiction shows that $\dim(S\cap G_\delta;G_\delta)=\w=\kappa$ for some $\delta\in\kappa^+$.
 \end{proof}
 
\begin{claim}\label{cl:delta2} There exists an ordinal $\delta\in\kappa^+$ such that $\dim(S\cap G_\delta;G_\alpha)=|S\cap G_\delta|=\kappa$ for every $\alpha\in[\delta,\kappa^+)$.
\end{claim}

\begin{proof} By Claim~\ref{cl:delta}, there exists an ordinal $\delta\in\kappa^+$ such that $\dim(S\cap G_\delta;G_\delta)=\kappa$ and hence $|S\cap G_\delta|=\kappa$. If $\kappa$ is uncountable, then for every $\alpha\in[\delta,\kappa^+)$ we have $\dim(S\cap G_\delta;G_\alpha)=|G\cap S_\delta|=\kappa$.

Next, assume that $\kappa=\w$. To derive a contradiction, assume that for some $\alpha\in[\delta,\kappa^+)$ we have $\dim(S\cap G_\delta;G_\alpha)<\w$ and take the smallest ordinal $\alpha$ with this property. It follows from $\dim(S\cap G_\delta;G_\delta)=\kappa$ that $\alpha>\delta$. The minimality of $\alpha$ ensures that $\alpha$ is a successor ordinal and hence $\alpha=\beta+1$ for some ordinal $\beta\in[\delta,\kappa^+)$. Now the inductive condition (vi) guarantees that $\dim(S\cap G_\delta;G_\beta)<\w$, which contradicts the minimality of $\alpha$. This contradiction shows that $\dim(S\cap G_\delta;G_\alpha)=\kappa$ for every $\alpha\in[\delta,\kappa^+)$.
\end{proof}

By Claim~\ref{cl:delta2} there exists an ordinal $\delta\in\kappa^+$ such that $\dim(S\cap G_\delta;G_\alpha)=|S\cap G_\delta|=\kappa$ for all $\alpha\in[\delta,\kappa^+)$. Let $\gamma\in\kappa^+$ be an ordinal such that $S\cap G_\delta=S_\gamma$. Since $|S|=\kappa^+$ and $G\setminus G_0=\bigcup_{\alpha\in\kappa^+}G_{\alpha+1}\setminus G_\alpha$, there exists an  ordinal $\alpha\in\kappa^+$ such that $\alpha\ge\max\{\delta,\gamma\}$ and $S\cap(G_{\alpha+1}\setminus G_\alpha)$ contains some element $a$.  The choice of the ordinal $\delta$ ensures that $\dim(S_\gamma;G_\alpha)=\kappa$. By the inductive condition (iv), there are elements $x_1,x_2,\dots,x_{18}\in S_\gamma\subseteq S$ such that $c=x_1ax_2a\cdots x_{18}a\subseteq S^{36}$. Therefore, $G=S^{36}$ and the group $G$ is $36$-Shelah.
\end{proof}

\begin{lemma}\label{l:mal} For every subset $X\subseteq G$ of cardinality $|X|<|G|$ there exists an element $g\in G$ such that $X\cap (gXg^{-1})\subseteq\{e\}$.
\end{lemma}

\begin{proof} Since $G=\bigcup_{\alpha\in\kappa^+}G_\alpha$ and $\mathrm{cf}(\kappa^+)=\kappa^+$, there exists an ordinal $\alpha\in \kappa^+$ such that $X\subseteq G_\alpha$. By the malnormality of $G_\alpha$ in $G_{\alpha+1}$, there exists an element $g\in G_{\alpha+1}\setminus G_\alpha$ such that $X\cap gXg^{-1}\subseteq G_\alpha\cap gG_\alpha g^{-1}\subseteq\{e\}$.
\end{proof}

\begin{lemma} The group $G$ is simple.
\end{lemma}

\begin{proof} Take any normal subgroup $N\ne\{e\}$ in $G$. Choose any element $a\in N\setminus\{e\}$ and find $\alpha\in\kappa^+$ such that $a\in G_\alpha$. For every ordinal $\beta\ge\alpha$ in $\kappa^+$  and any $b\in G_{\beta+1}\setminus G_\beta$ we have $bab^{-1}\in N\setminus G_\beta$ by the malnormality of $G_\beta$ in $G_{\beta+1}$. Therefore, $|N|=\kappa^+$ and $N=N^{36}=G$ by the $36$-Shelah property of the group $G$, proved in Lemma~\ref{l:36}.
\end{proof}

\begin{lemma}\label{l:cov1} For every subset $X\subseteq G$ of cardinality $|X|<|G|$ there exists an element $\lambda\in G$ of order $51$ such that $X\subseteq \{x\in G:\lambda x\lambda^2x\lambda^3x\cdots \lambda^{18}x=e\}$ and hence $\cov(X;\A_G)\le1$.
\end{lemma}

\begin{proof} By the regularity of the cardinal $\kappa^+=|G|$, there exists  an ordinal $\alpha\in\kappa^+$ such that $X\subseteq G_\alpha$. By the inductive condition (iii), there exists an element $\lambda\in G_{\alpha+1}\setminus G_\alpha$ of order 51 such that $X\subseteq G_\alpha\subseteq \{x\in G:\lambda x\lambda^2x\lambda^3\cdots\lambda^{18}x=e\}$ and hence $\cov(X;\A_G)\le 1$.
\end{proof}

\begin{lemma}\label{l:non-poly} For every subset $C\subseteq G$ of cardinality $|C|<\mathrm{cf}(\kappa)$ there exists an element $x\in G$ such that $e\ne xc_1xc_2\cdots xc_n$ for all $c_1,\dots, c_n\in C$.
\end{lemma}

\begin{proof} By the regularity of the cardinal $\kappa^+$, there exists an ordinal $\alpha\in\kappa^+$ such that $C\subseteq G_\alpha$. By the inductive condition (v), there exists an element $x\in G_{\alpha+1}\subseteq G$ such that $e\ne xc_1\cdots xc_n$ for all elements $c_1,\dots, c_n\in C$.
\end{proof}

\begin{lemma}\label{l:pd} The group $G$ is projectively $\mathsf{T_{\!1}S}$-discrete.
\end{lemma}

\begin{proof}  Given any homomorphism $h:G\to Y$ to a $T_1$ topological semigroup $Y$, we should prove that the image $h[G]$ is a discrete subspace of $Y$. Since the group $G$ is simple, the homomorphism $h$ is either constant or injective. If $h$ is constant, then the space $h[G]$ is discrete, being a singleton. So, assume that $h$ is injective. Let $\tau$ be a unique topology on $G$
such that the bijective homomorphism $h:(G,\tau)\to h[G]$ is a homeomorphism. It follows that $\tau$ is a $T_1$ semigroup topology on $G$. Choose any element $c\in G\setminus\{e\}$ and using the continuity of the group multiplication in the topological semigroup $(G,\tau)$, find a neighborhood $U\in\tau$ of
$e$ such that $c\notin U^{36}$. Since the group $G$ is $36$-Shelah, $|U|\le\kappa$ and by Lemma~\ref{l:mal}, there exists $x\in G$ such that $U\cap xUx^{-1}\subseteq\{e\}$, which means that the identity $e$ of $G$ is an isolated point
of the topological space $(G,\tau)$. The topological homogeneity of the paratopological group
$(G,\tau)$ implies that the topological space $(G,\tau)$ is discrete and so is its homeomorphic copy $h[G]\subseteq Y$.
\end{proof}

\begin{lemma}\label{l:ac} The group $G$ is absolutely $\mathsf{T_{\!1}S}$-closed 
 and projectively $\mathsf{T_{\!1}S}$-discrete.
\end{lemma}

\begin{proof} Assuming that the group $G$ is not absolutely $\mathsf{T_{\!1}S}$-closed, we can find a homomorphism $h:G\to Y$ to a topological semigroup $Y\in\mathsf{T_{\!1}S}$ whose image $h[G]$ is not closed in $Y$. Since the group $G$ is simple, the homomorphism $h$ is injective. Since the group $G$ is projectively $\mathsf{T_{\!1}S}$-discrete, the image $h[G]$ is a discrete subspace of $Y$. Chose any point $y\in\overline{h[G]}\setminus h[G]$. Let $c\in G$ be any point such that $h(c)\ne y^{36}$. By the continuity of the semigroup operation in the $T_1$ topological semigroup $Y$, there exists a neighborhood $U\subseteq Y$ of  $y$ such that $h(c)\notin U^{36}\subseteq Y$. Consider the set $V=h^{-1}[U]$ and observe that $c\notin V^{36}$. Since the group $G$ is 36-Shelah, the set $V$ has cardinality $|V|<|G|=\kappa^+$.
By Lemma~\ref{l:cov1}, the set $V$ is polybounded in $G$ and then the set $h[V]$ is polybounded in $h[X]$. By Lemma~\ref{l:polybounded}, the set $h[V]$ is closed in $Y$, which is not possible as $y\in \overline{h[V]}\setminus h[V]$. This contradiction shows that the group $G$ is absolutely $\mathsf{T_{\!1}S}$-closed. By Lemma~\ref{l:ac=>pd}, the absolutely $\mathsf{T_{\!1}S}$-closed group $G$ is projectively $\mathsf{T_{\!1}S}$-discrete.
\end{proof}

\begin{lemma} $ \mathrm{cf}(\kappa)\le\cov(G;\A_G)\le\kappa$ and hence the group $G$ is not polybounded.
\end{lemma}

\begin{proof} Let $\gamma=\cov(G;\A_G)$. By the definition of the cardinal $\cov(G;\A_G)$, there exists a cover $\{A_\alpha\}_{\alpha\in\gamma}$ of $G$ by $\gamma$ many algebraic sets. For every $\alpha\in\gamma$, find $b_\alpha\in G$ and a semigroup polynomial $p_\alpha:G\to G$ such that $A_\alpha=p_\alpha^{-1}(b_\alpha)$. By the definition of a semigroup polynomial, there exist elements $c_{\alpha,0},c_{\alpha,1},\cdots c_{\alpha,n_\alpha}\in G$ such that $p_\alpha(x)=c_{\alpha,0}xc_{\alpha,1}\cdots xc_{\alpha,n_\alpha}$ for all $x\in G$. Replacing the last coefficient $c_{\alpha,n_\alpha}$ by $c_{\alpha,n_\alpha}b_\alpha^{-1}c_{0,\alpha}^{-1}$, we can assume that $b_\alpha=e=c_{\alpha,0}$ and hence $G=\bigcup_{\alpha\in\gamma}p_\alpha^{-1}(e)$. Consider the set $C=\bigcup_{\alpha\in\gamma}\{c_{\alpha,1},\cdots,c_{\alpha,n_\alpha}\}$ and observe that $|C|\le\gamma\cdot\w$. Assuming that $\gamma=\cov(G;\A_G)<\mathrm{cf}(\kappa)$, we conclude that $|C|<\mathrm{cf}(\kappa)$. By Lemma~\ref{l:non-poly}, there exists $x\in G$ such that $e\ne xc_1\cdots xc_n$ for all elements $c_1,\dots,c_n\in C$. Find $\alpha\in\gamma$ such that $x\in A_\alpha$ and conclude that 
$$e=b_\alpha=p_\alpha(x)=xc_{\alpha,1}xc_{\alpha,2}\cdots xc_{\alpha,n_\alpha}\ne e.$$ This contradiction proves that $\cov(G;\A_G)\ge\mathrm{cf}(\kappa)\ge\w$.  The upper bound $\cov(G;\A_G)\le\kappa$ follows from Lemma~\ref{l:ac} and Lemma~\ref{l:zero-closed}. 
\end{proof}
\end{proof}

\section{Some remarks and open problems}\label{s:final}

By Theorem~\ref{t:main}, under Continuum Hypothesis there exists an uncountable $36$-Shelah group. On the other hand, we have the following result, proved by Yves Cornullier in his answer to the question \cite{MO} of the author on {\tt MathOverflow}.

\begin{theorem}[Cornullier] Every $3$-Shelah group is finite.
\end{theorem}

\begin{question} Is every $4$-Shelah group finite?
\end{question}

Also Theorem~\ref{t:main} suggests the following natural problem, posed also in  \cite{MO}

\begin{problem} Can an infinite Shelah group be constructed in ZFC?
\end{problem}

\end{document}